\documentclass[11pt]{article}
\usepackage[margin=3cm]{geometry}
\usepackage{amsmath,amssymb,amsthm}
\usepackage{bbm}
\usepackage{hyperref}
\date{\today}

\newtheorem{theorem}{Theorem}[section]
\newtheorem{thm}{Theorem}[section]
\newtheorem{proposition}[thm]{Proposition}

\newtheorem{corollary}[thm]{Corollary}

\newtheorem{lemma}[thm]{Lemma}
\newtheorem{definition}[thm]{Definition}

\newtheorem{remark}[thm]{Remark}

\renewenvironment{proof}{\noindent{\bf Proof:}}{\hfill$\square$\vskip.5cm}
\newenvironment{proofof}{\noindent}{\hfill$\square$\vskip.5cm}

\newcommand{\Prob} {{\bf P}}
\newcommand{\Z}{\mathbb {Z}}
\newcommand{\E}{{\bf E}}

\newcommand{\R}{{\mathbb{R}}}

\newcommand{\Q}{{\bf Q}}

\newcommand {\og} {\overline{\gamma}}
\newcommand {\ogp} {\overline{\gamma}^{\prime}}

\newcommand {\tg} {\tilde{\gamma}}

\newcommand{\tProb} {{\bf \tilde{P}}}
\newcommand{\T} {\mathbb{T}_L^{d-1}}
\newcommand{\tT} {\tilde{T}}
\newcommand{\hg}{\hat{\gamma}}

\newcommand{\U}{\mathcal{U}}

\newcommand{\oA}{\overline{\mathcal{A}}}

\title{Intersection exponents for biased random walks \\ on discrete cylinders}

\author{Brigitta Vermesi\\
University of Rochester}

\date{September 19, 2008}

\begin{document}
\maketitle

\begin{abstract}
We prove existence of intersection exponents $\xi(k,\lambda)$ for
biased random walks on $d$-dimensional half-infinite discrete
cylinders, and show that, as functions of $\lambda$, these exponents
are real analytic. As part of the argument, we prove convergence to
stationarity of a time-inhomogeneous Markov chain on half-infinite
random paths. Furthermore, we show this convergence takes place at
exponential rate, an estimate obtained via a coupling of weighted
half-infinite paths.

\medskip
\noindent {\bf Keywords:} Intersection exponent, biased random walk,
coupling
\end{abstract}

\section{Introduction}

In this paper, we analyze biased random walks on $d$-dimensional
discrete cylinders. The treatment of the model is self contained and
it does not use previous results.

Our main motivation in approaching this problem is to understand the
arguments and techniques needed in the study of $3$-dimensional
Brownian intersection exponents. In a series of papers, Lawler,
Schramm and Werner studied Brownian intersection exponents in
dimensions two and three. They also proved that, in dimension two,
these exponents are analytic. We think their analyticity argument
will apply in three dimensions, as long as one is able to get good
estimates on the coupling rate of weighted Brownian motion paths. In
two dimensions, the estimates are obtained using conformal
invariance, and so they do not transfer to three dimensions. By
analyzing a transient random walk on the cylinder, we will be able
to highlight the techniques needed to prove analyticity of
intersection exponents for Brownian motion, as well as the type of
estimates needed for an exponential coupling rate.

Let us note that analyzing random walks on cylinders has been of
great interest in recent years. We direct the reader to the work of
Sznitman and Dembo on the disconnection of cylinders by random
walks, such as \cite{Dembo} and \cite{Sznitman}. The recent work of
Windisch \cite{Windisch} on the disconnection time of discrete
cylinders by biased random walks, where the drift depends on the
size of the base, is also worth noting.

\subsection{Brownian intersection exponents}
Let us begin with a brief introduction to Brownian intersection
exponents. Consider a set of $k$ independent Brownian motion paths
started at the origin and a set of $j$ independent Brownian motion
paths started away from the origin, on the ball of radius one. The
probability that the two packets reach level $e^n$, without
intersecting, decays exponentially in $n$ with exponent
$\xi^{(BM)}(k,j)$. Roughly speaking, intersection exponents for
Brownian motion are a measure of how likely it is that Brownian
motion paths do not intersect. We now proceed to make this precise.

For $d=2, 3$, $B^1_t, B^2_t,...,B^k_t$ will denote $k$ independent
$d$-dimensional Brownian motions started at the origin. Let $X_t$ be
another $d$-dimensional Brownian motion started on the ball of
radius $1$ and independent of $B^1,\cdots,B^k$. For $1\le i\le k$ we
write $B^i[0,t]:=\{z\in \R^d: B^i_s=z \mbox{ for some } 0\le s\le
t\}$ and similarly $X[0,t]:=\{z\in \R^d: X_s=z \mbox{ for some }
0\le s\le t\}$.

For $1\le j\le k$, let
\[\tilde{T}^j_n=\inf\{t: |B^j_t|\ge e^n\}\]
be the first time $B^j$ reaches the ball of radius $e^n$. Similarly,
let
\[\tilde{T}_n=\inf\{t: |X_t|\ge e^n\}.\]

 Given $k$ paths, we let $Z_n^{(BM)}$ be the
probability that another path avoids them, up to the first time both
sets reach the ball of radius $e^n$, as follows:
\[Z_n^{(BM)}=\Prob\left\{X[0,\tilde{T}_n]\cap (B^1[0,\tilde{T}^1_n]\cup\cdots
\cup B^k[0,\tilde{T}^k_n])=\emptyset \; \vert\;
B^1[0,\tilde{T}^1_n]\cup\cdots\cup B^k[0,\tilde{T}^k_n]\right\}.\]

Then the intersection exponent $\xi^{(BM)}(k,j)$ is defined as
\[\xi^{(BM)}(k,j):=-\lim_{n\to
\infty}\frac{\log\E[(Z^{(BM)}_n)^j]}{n}.\]

One can further define generalized intersection exponents that
loosely speaking describe non-intersection probabilities between
non-integer numbers of Brownian paths. They were first introduced in
\cite{LW}. In other words, the discrete sequence of intersection
exponents can be replaced in a natural way by a continuous function
\[\xi^{(BM)}(k,\lambda):=-\lim_{n\to
\infty}\frac{\log\E[(Z^{(BM)}_n)^{\lambda}]}{n}.\] The existence of
intersection exponents follows from a subadditivity argument.
Alternate but equivalent ways of defining such exponents can be
found in \cite{Fractal Prop}, \cite{Strict Concavity}.

Brownian intersection exponents have been studied extensively. In
dimensions four and higher, since two Brownian paths do not
intersect almost surely, Brownian intersection exponents in these
dimensions equal zero. Presently, all intersection exponents for the
planar Brownian motion are known (see \cite{LSW1}, \cite{LSW2},
\cite{LSW3}): intersection exponents for a wide range of values of
$\lambda$ have been computed using the Schramm-Loewner evolution
($SLE$). Lawler, Schramm and Werner further proved that planar
Brownian intersection exponents are analytic \cite{Anal Paper} and
consequently $\xi^{(BM)}_2(k,\lambda)$ were extended by analyticity
to $\lambda>0.$ However, not much progress has been made in three
dimensions. As of this moment, the only known exponents are
$\xi^{(BM)}_3(k,0)=0$ and $\xi^{(BM)}_3(2,1)=\xi^{(BM)}_3(1,2)=1$.
In \cite{BL}, it was proved that the exponent $\xi^{(BM)}_3(1,1)$ is
between $1$ and $1/2$ and \cite{Strict Concavity} implies it is
strictly between $1$ and $1/2$. Simulations further suggest this
exponent is around $.58$ (see \cite{Simulation}).

\subsection{Summary of results}
Let $G$ be the half-infinite discrete cylinder $\Z\times \T$, where
$\T$ is a $(d-1)$-dimensional torus of side $L$. We consider random
walks on $G$ that move according to the following transition
probabilities:

\begin{equation}\label{rw_transitions}
p(z,w)=\left\{
\begin{array}{ll}
p/d & \mbox{if } w-z=(1,0)\\
(1-p)/d & \mbox{if } w-z=(-1,0)\\
1/(2d) & \mbox{if } w-z=(0,\vec{1})\\
0 & \mbox{otherwise}
\end{array}\right.
\end{equation}
where $\vec{1}$ denotes any vector of norm one on $\T$. Here the
second coordinate is a $d$-dimensional vector and addition is, as
usual, $\hspace{-.1in}\mod L$ on $\T.$ Then the random walk is
symmetric
 on $\T$, and it is a one-dimensional asymmetric random walk on
$\Z$ with parameter $p>1/2$.

The path-valued random variables $Z_n$ are defined for this process
in the same manner we defined them for Brownian motion. Roughly
speaking, they are the probability that, given a set of $k$
half-infinite paths up to level $n$, another random walk coming from
negative infinity will reach level $n$ without hitting the given set
of paths. In order to give a precise definition for $Z_n$, we
introduce a class of paths that we will call \emph{nice}; these
paths have the property that they can be avoided by a random walk
coming from negative infinity and, in particular, they do not
disconnect negative infinity and level zero. We note that the
definition of \emph{nice} paths is such that two $h$-processes
conditioned to avoid a given path can be coupled (see Section
\ref{nice}). One can use this to describe the hitting measure on any
level of a random walk started at negative infinity and conditioned
to avoid a given random walk path. A similar result for
3-dimensional Brownian motion has not yet been proved but it is
expected to be true. In particular, one should be able to describe
the hitting measure on the ball of radius one for a Brownian motion
started close to another Brownian motion and conditioned to avoid
it.

The random walk intersection exponent $\xi(k,\lambda)$ is defined as
\[\xi(k,\lambda):=-\lim_{n\to
\infty}\frac{\log\E[Z_n^{\lambda}]}{n}.\] In Section
\ref{sec:exponents} we use a subadditivity argument to show these
exponents exist and are finite, with
$\displaystyle{\E[Z_n^{\lambda}]}$ being logarithmically asymptotic
to $\displaystyle{e^{-\xi(k,\lambda)n}}$. As long as we start with a
\emph{nice} initial configuration, we further prove an important
estimate: $\E[Z_n^{\lambda}]$ are within constant multiples of
$e^{-\xi(k,\lambda) n}$ (see Proposition \ref{estconst}), which we
will denote by \[\E[Z_n^{\lambda}]\asymp e^{-\xi(k,\lambda) n}.\]

As mentioned before, one of our main tools is coupling of weighted
paths. Starting with an initial configuration $\og_0$, we attach to
it a random walk started at the endpoint of $\og_0$ and stopped when
it first reaches level one. Call the resulting path $\og_1$. We will
fix a large $N$ and condition the path $\og_1$ to survive up to
level $N$, that is, $N-1$ additional steps. Then we weight the new
path by the probability it will survive up to level $N$, and
normalize it to obtain a probability measure. This procedure defines
a time-inhomogeneous Markov chain $X_n$ on the space of \emph{nice}
paths, depending on the initial configuration $\og_0$. One of our
main results is the following theorem, whose proof is the content of
Section \ref{sec:Exp Convergence}:

\begin{theorem}\label{expc} Let $X_n$ and $X_n'$ be Markov chains with
$X_0=\og_0$, and $X'_0=\ogp_0$ respectively, induced by the
weighting described above. There exist constants $C,\beta>0$ such
that, for all $n\ge 1$, for all $\og_0,\ogp_0 \in \oA$, we can
define $X_n$ and $X'_n$ on the same probability space
$(\tilde{\Omega},\tilde{\mathcal{F}},\tilde{\Prob})$ with
\[\tilde{\Prob}\{X_n\not\equiv_{n/2}X'_n\}\le Ce^{-\beta n},\]
where $X_n\equiv_{k}X'_n$ means that $X_n$ and $X'_n$ have been
coupled for the past $k$ steps.
\end{theorem}

Let us briefly describe how we will couple two Markov chains started
with different initial configurations. The coupling is similar to
the coupling used in \cite{Coupling Paper} for a one-dimensional
Ising-type model. The maximal coupling is done only on the set of
paths that have \emph{enough} connected cross-sections and the other
transition probabilities are then adjusted so that we obtain a
probability measure on the set of \emph{nice} paths. Once the two
chains are coupled, they do not necessarily remain coupled; in fact,
if at the next level the paths do not have \emph{enough} connected
cross-sections they will decouple. However, if the two chains are
coupled, the probability they will remain coupled for an additional
step increases with the number of steps for which they have been
already coupled. Once the coupling is set up, we can use roughly the
same argument as in \cite{Coupling Paper} to prove Theorem
\ref{expc}.

As a result of this coupling, one can show that starting with a
given measure $\nu$ on initial configurations $\og_0$, supported on
$\oA$, if we let the process evolve, then the measure induced by
this process on half-infinite paths converges to an invariant
measure $\pi$.
\begin{theorem}\label{invariant}
Let $\nu$ be a measure supported on $\oA$ and let $\nu^n$ be the
measure on paths in $\oA$, whose density with respect to $\nu$ is
\begin{equation*} \frac{Z_n^{\lambda}}{\E^{\nu}[Z_n^{\lambda}]}.
\end{equation*}
Then there exists a measure $\pi$ supported on $\oA$ such that
$\nu^n$ converges to $\pi.$ Furthermore, if $\pi^n$ has density
$\displaystyle{\frac{Z_n^{\lambda}}{\E^{\pi}[Z_n^{\lambda}]}}$ with
respect to the limiting measure $\pi$, then $\pi^n=\pi.$
\end{theorem}
The proof of this theorem is the content of Section
\ref{sec:Invariant Measure}. Our main result can be found in Section
\ref{sec:anal}. Using the same ideas and techniques that Lawler,
Schramm and Werner used in \cite{Anal Paper}, we prove analyticity
of the intersection exponent $\xi(k,\lambda)$.

\begin{theorem} \label{anal_proof} For all $k\ge 1$, $\xi(k,\lambda)$ is a real
analytic function of $\lambda$ in $(0,\infty).$
\end{theorem}
The argument has similarities with the proof from \cite{Ruelle} that
the free energy of a one-dimensional Ising model with exponentially
decreasing interactions is an analytic function. The proof follows
the structure of the proof for analyticity of planar Brownian
intersection exponents, presented in \cite{Anal Paper}, and it
differs from \cite{Anal Paper} only in the estimates we use. In
fact, this suggests that, if similar estimates can be obtained for
3-dimensional Brownian exponents, the analyticity proof from
\cite{Anal Paper} would immediately apply.

Here is a brief outline of the proof. We will restrict ourselves to
proving analyticity for $\xi(\lambda):=\xi(1,\lambda).$ We define
linear functionals $T_{\lambda}$ on a Banach space of functions,
defined on the set of \emph{nice} paths, functions with the property
that they depend very little on how the path looks like far away in
the past. $T_{\lambda}$ and the Banach space are described in
Section \ref{Bspace}. The norm on the Banach space is a little
different than the one in \cite{Anal Paper}, but it follows the same
principle. We first show $T_z$ is an analytic function in a
neighborhood of the positive real line in Section \ref{Tanal}. The
existence of the spectral gap is done in Section
\ref{sec:analyticity} by showing the exponential coupling rate
implies $e^{-\xi(\lambda)}$ is an isolated simple eigenvalue of
$T_{\lambda}$. The theorem then follows by a simple argument from
operator theory.

\subsection{A word on notation} Throughout the article, we will say a
sequence $a_n$ is \emph{logarithmicaly asymptotic} to $e^{bn}$,
which we will write as $a_n\approx e^{bn}$ if
\[\lim_{n\to\infty}\frac{\log a_n}{n}=b.\]
Also we will use the notation $a_n\asymp b_n$ to denote the
following: there exist constants $c$ and $C$ such that for all $n$,
\[c b_n\le a_n\le C b_n.\]
We will often use the notation $a_n=O(b_n)$ by which we mean that
there exist constants $C, N>0$ such that for all $n\ge N$,
\[a_n\le C\, b_n.\]

Constants $c$, $c'$ and $C$ will denote arbitrary positive
constants, independent of all other quantities involved in a given
expression. Their value will be allowed to change from line to line.
However, other constants, such as $a, \hat{c}, c_1, c_2$  will be
fixed.

\subsection*{Acknowledgements} The author would like to thank
Gregory Lawler for suggesting this problem and for many useful
conversations and advice throughout the completion of this project.

\section{Random walk intersection exponents}

\subsection{Paths on the cylinder}
Let $\displaystyle{\T=(\Z/L\Z)^{d-1}}$ be a $(d-1)$-dimensional
torus and let \[G:=\Z\times\T, \hspace{.3in} d\ge 1,\] be the
discrete $d$-dimensional infinite cylinder. The set $\{(i,y):y\in
\T\}$ will be called \emph{the level $i$} of the cylinder, and for a
point $z=(i,y)\in G$ we write $|z|=i$ to mean $z$ is on level $i$ of
the cylinder. Our main motivation in using the torus as a base for
the cylinder is the property that every point on the torus ``looks
the same.'' Therefore, one could equally well consider a finite
connected regular graph as a generalization of $\T$.

For $i\le j$, denote by $G_{i,j}$ the cylinder between levels $i$
and $j$,
\[G_{i,j}=\{z\in \Z\times\T :i\le |z|\le j\}.\]
We will write $G_j=G_{-\infty,j}$ for the half-infinite cylinder up
to level $j$.

We construct half-infinite paths on $G$ as follows. Let
$\mathcal{X}$ be the set of all paths $\gamma:[0,t_{\gamma}]\to G$,
starting on level 0 and ending when first reaching level 1:
\[\mathcal{X}:=\left\{\gamma:[0,t_{\gamma}]\to G :
|\gamma(0)|=0, |\gamma(t_{\gamma})|=1 \mbox{ and } |\gamma(t)|<1
\mbox{ for all } t<t_{\gamma}\right\}.\]

We will refer to $t_{\gamma}$ as the time-duration of the path
$\gamma$ and we will say $\gamma$ and $\gamma'$ are \emph{equal} if
$\gamma'$ is a translation of $\gamma$ in the $\T$ direction of $G$.
Let $\mathcal{X}_i$ be the set of all paths $\gamma_i$ starting on
level $(i-1)$ and stopped when they reach level $i$. Then
$\mathcal{X}_i$ is exactly $\mathcal{X}$ translated by $(i-1)$ in
the $\Z$ direction of $G$. It will be convenient to think of
elements $\gamma_i$ of $\mathcal{X}_i$ as paths in $\mathcal{X}$,
translated accordingly. Let
\[\mathcal{A}=\{ \overline{\gamma}_0=\dotsc\gamma_{-2}\gamma_{-1}\gamma_0 : \gamma_j \in \mathcal{X}_j \mbox{ and } \gamma_{j}(t_{\gamma_{j}})=\gamma_{j+1}(0) \mbox{ for } j< 0\}.\]
That is, $\mathcal{A}$ is the set of half-infinite paths $\og_0$
constructed as a sequence of $\gamma_j \in \mathcal{X}_j$, $-\infty<
j\le 0$. We will think of $\mathcal{A}$ as $\cdots \mathcal{X}\times
\mathcal{X}$. More precisely, an element of $\cdots
\mathcal{X}\times\mathcal{X}$ along with a position at time zero
uniquely determines a path in $\mathcal{A}$.

Each $\og_0$, as a sequence of paths, is in a one-one correspondence
with a path indexed by time $\og_0(t): (-\infty, 0]\rightarrow G$.
The construction of $\og_0(t)$ from $\og_0$ is left as a simple
exercise. We will write $\og_0$ when we look at the path as a
sequence of elements of $\mathcal{X}$, and we will write $\og_0(t)$
when we need to look at $\og_0$ as a sequence of points in the
half-infinite cylinder.

For $n\in\Z$, let $\mathcal{A}_n$ be the set $\mathcal{A}$ shifted
by $n$ in the $\Z$ coordinate of $G$, and denote its elements by
$\og_n$. Observe that for $n\ge 1$, $\og_n$ can be decomposed into
$\og_n=\og_0 \gamma_1\cdots \gamma_n$, for some unique $\og_0\in
\mathcal{A}$ and $\gamma_j\in \mathcal{X}_j$ for $1\le j\le n$. Let
$\hg_n=\gamma_1\cdots\gamma_n$.
\begin{definition}
Let $\og_n=\dotsc\gamma_{n-1}\gamma_n$ and
$\ogp_n=\dotsc\gamma^{\prime}_{n-1}\gamma^{\prime}_n$ be paths in
$\mathcal{A}_n$. We write
\[\og_n=_k\ogp_n\]
if $\gamma_j=\gamma^{\prime}_j$ for all $n-k+1\le j\le n$, and we
say $\og_n$ and $\ogp_n$ agree for the last $k$ levels.
\end{definition}
The half-infinite paths $\og_0$ that we will be studying are biased
random walks coming from negative infinity.

\subsection{Intersection exponent - existence}\label{sec:exponents}

Let us consider $k$ independent random walks $S^1, \dots, S^k$
defined on the probability space $(\Omega, \mathcal{F},\Prob)$,
starting on level zero of $G$, and evolving according to transition
probabilities as in \eqref{rw_transitions}. Let $S$ be another
random walk, defined on the probability space $(\Omega_1,
\mathcal{F}_1, \Prob_1)$, with the same transition probabilities. We
will use $\E$ and $\E _1$ for expectations with respect to $\Prob$
and $\Prob_1$, respectively.
\begin{remark}
Observe that we define $S$ on $(\Omega_1, \mathcal{F}_1, \Prob_1)$
and $S^1,\dots,S^k$ on $(\Omega, \mathcal{F}, \Prob)$. This notation
may look unnatural, but it will help simplify notation later in the
paper.
\end{remark}
 Define stopping times: for $1\le j\le k$,
\[T_r^j=\inf\{t:|S_t^j|=r\}\]
\[T_r=\inf\{t:|S_t|=r\}.\]
$S[r,s]$ will denote the random-valued set $\{S_l: r\le l \le s\}$,
the set of points on the cylinder visited by the random walk from
time $r$ to time $s$. Similarly, for $1\le j\le k$,
$S^j[r,s]=\{S^j_l: r\le l \le s\}$.

We start with a $k$-tuple $\Gamma_0:=(\og_0^1,\dots,\og_0^k)$ of
$\mathcal{A}^k$, which will be called an initial configuration. For
$1\le j\le k$, let $S^j$ be a random walk on the cylinder $G$,
started at the endpoint of $\og_0^j$, and evolving, independent of
$\og_0^j$, according to transition probabilities in
\eqref{rw_transitions}. Take the path $S^j$ stopped when it first
reaches level $n$ and attach it to $\og_0^j$. This is a path from
$-\infty$ to level $n$, and in particular it is an element of
$\mathcal{A}_n$. We denote it by $\og_n^j$. Let $\tg_n^j$ be the
path $\og_n^j$ shifted accordingly in the $\Z$ direction of $G$ so
that it is an element of $\mathcal{A}$.

For all $n\in\Z$, define $\Gamma_n:=(\og_n^1,\dots,\og_n^k)$ and
$\tilde{\Gamma}_n:=(\tg_n^1,\dots,\tg_n^k)$. Let $\mathcal{F}_n$ be
the $\sigma$-algebra generated by $\Gamma_0$ and the random walks
$S^1,\dots,S^k$ up to stopping times $T_n^1,\dots,T_n^k$
\[\mathcal{F}_n=\sigma\{\Gamma_0, S_t^j; t\le T_n^j, \mbox{ for $1\le j\le k$}\}.\]
Then $\Gamma_n$ is an $\mathcal{F}_n$-measurable (path-valued)
random variable.

We define functions $Z_n: \mathcal{A}^k\to\R$ by
\begin{equation*}
Z_n(\Gamma_0)=\Prob_1\{S(-\infty,T_0]\cap \Gamma_0=\emptyset|
S(-\infty,T_{-n}]\cap \Gamma_{-n}=\emptyset\}.
\end{equation*}
Equivalently, one can define $Z_n$ as
\begin{equation}\label{defZ}
Z_n(\tilde{\Gamma}_n)=\Prob_1\{S(-\infty,T_n]\cap
\Gamma_n=\emptyset| S(-\infty,T_0]\cap \Gamma_0=\emptyset\}.
\end{equation}
More precisely,
\begin{equation}\label{eq:nice_A}
Z_n(\tilde{\Gamma}_n)=\lim_{|x|\to-\infty}\Prob_1^x\{S[0,T_n]\cap
\Gamma_n=\emptyset| S[0,T_0]\cap \Gamma_0=\emptyset\},
\end{equation}
where the initial configuration $\Gamma_0$ is \emph{nice}, meaning
that $\Gamma_0$ is such that this limit exists. Note also that we
write conditioning with respect to the event
\[S(-\infty,T_0]\cap \Gamma_0=\emptyset,\]
which is a set of probability zero. By this conditioning we mean
that on $(-\infty, T_0]$, $S$ is an $h$-process conditioned not to
hit $\Gamma_0$. Given $\Gamma_0$ is \emph{nice}, the conditioning is
well-defined and the limit exists, as it will be discussed in
Section \ref{nice}. For now, let us assume $Z_n$ are well-defined.

Let us also consider a random walk started at $z$, on level zero,
and define the following $\mathcal{F}_n$-measurable random variables
 \[\overline{Z}_{n,z}=\Prob_1^{z}\{S[0,T_n]\cap \Gamma_n= \emptyset\}\]
\[\overline{Z}_n=\sup_{|z|=0}\Prob_1^{z}\{S[0,T_n]\cap \Gamma_n= \emptyset\}.\]
Let $\displaystyle{q_n=\sup_{\Gamma_0 \in
    \mathcal{A}^k}\E^{\Gamma_0}[Z_n^{\lambda}]}$
and $\displaystyle{\overline{q}_n=\sup_{\Gamma_0 \in
    \mathcal{A}^k}\E^{\Gamma_0}[\overline{Z}_n^{\lambda}]}$ .
Now,
\begin{eqnarray*}
Z_n&=&\Prob_1\{S(-\infty,T_n]\cap \Gamma_n=\emptyset|
S(-\infty,T_0]\cap
    \Gamma_0=\emptyset\}\\
&\le& \sup_{|z|=0}\Prob_1^z\{S[0,T_n]\cap
\Gamma_n=\emptyset\}=\overline{Z}_n
\end{eqnarray*}
and taking expectations, we have $q_n\le \overline{q}_n$.

\begin{proposition}
There exists $\xi(k,\lambda)$ such that $q_n\approx
e^{-n\xi(k,\lambda)}$, that is
\[\lim_{n\to
    \infty}\frac{\log{q_n}}{n}=-\xi(k,\lambda)\]
\end{proposition}

\begin {proof}
Let $\oA^k$ be the set of all \emph{nice} $k$-tuples $\Gamma_0$. For
$n\ge 1$, define functions $\Phi_n:\oA^k\to \R$ by
\[\Phi_n(\Gamma_0)=-\log Z_n(\Gamma_0).\]
We will use $\Phi$ to denote $\Phi_1$ and use the shorthand $\Phi_m$
for $\Phi_m (\Gamma_0)$. We naturally let $\Phi_0=0$. It is easy to
see that the functions $\Phi_m$ are additive, more precisely
\[\Phi_{n+m}(\tilde{\Gamma}_{n+m})=\Phi_n(\tilde{\Gamma}_{n})+\Phi_m(\tilde{\Gamma}_{n+m}).\]
Then we have
\begin{eqnarray*}
q_{m+n}&=& \sup_{\Gamma_0}\E^{\Gamma_0}[e^{-\lambda \Phi_{m+n}}]\\
&=& \sup_{\Gamma_0}\E[\E^{\Gamma_0}[e^{-\lambda
\Phi_{m+n}(\tilde{\Gamma}_{n+m})}|\mathcal {F}_n]]\\
&=& \sup_{\Gamma_0}\E^{\Gamma_0}[e^{-\lambda \Phi_n}
\E^{\tilde{\Gamma}_n}[e^{-\lambda\Phi_m}]]\\
&\le&\sup_{\Gamma_0}\E^{\Gamma_0}[e^{-\lambda \Phi_n}q_m]\\
&\le& q_m q_n.
\end{eqnarray*}
Then $\log{q_{n+m}}\le\log{q_n}+\log{q_m}$, and using an easy
subadditivity argument (see \cite{RW Lawler}), we get
\begin{equation}\label{existence}
\lim_{n\to
    \infty}\frac{\log{q_n}}{n}=\inf_{n}\frac{\log{q_n}}{n}=-\xi(k,\lambda),
\end{equation}
with $\xi(k,\lambda)$ possibly infinite. To see that $\xi(k,\lambda)
<\infty$, suppose $S$ has avoided $\Gamma_0$ up to time $T_0$. Let
the paths given by $S[T_0,T_n]$ and $S^1[0,T^1_n],\dots,
S^k[0,T^k_n]$ be straight lines in the $\Z$ direction of $G$. Then
$\Gamma_n\cap S(-\infty,T_n]=\emptyset$. This configuration occurs
with probability $\left(p/d\right)^{(k+\lambda)n}$, if $k<|\T|-1$
and with probability at least $\left(p/d\right)^{(|\T|-1+\lambda)n}$
if $k\ge|\T|-1$. Therefore $\xi(k,\lambda)\le (|\T|-1+\lambda)\log
(d/p)$.
\end{proof}

From now on we will only consider the case $k=1$ and analyze the
exponent $\xi(\lambda):=\xi(1,\lambda)$. Proofs for $k>1$ are
essentially the same.

\subsubsection{Nice paths}\label{nice}
Recall definition \eqref{eq:nice_A} of $Z_n$. In this section we
will present the technicalities involved in making sense of this
definition. The reader is welcome to skip this section at a first
reading.

Let $\og_0$ be a path in $\mathcal{A}$. Let $D(\og_0)$ be the
connected component of $G_0\setminus \og_0$ connecting level 0 to
negative infinity. Of course, $\og_0$ might disconnect level 0 from
negative infinity, in which case $D(\og_0)=\emptyset.$ If $D(\og_0)$
is non-empty, then it is unique for the following reason: $\og_0$
has only one point on level zero, so level zero is connected, and
hence there is at most one connected component containing $-\infty$
and level zero minus $\og_0$. For $j\le 0$, let $D_j=G_{j,j}\cap
D(\og_0)$ be the set of sites on level $j$ that can be reached by a
random walk from $-\infty$, conditioned to avoid $\og_0$.

\begin{definition}\label{def:nice}
$\og_0$ is \emph{nice} if $D(\og_0)\neq \emptyset$ and for all $n\ge
0$, there exists a $k>n$ such that
\[\sum_{j=-k}^{-1} 1_{\left\{D_j \mbox{ is connected}\right\}}\ge
n
\]
Let $\oA$ be the set of all such \emph{nice} paths.
\end{definition}
This definition simply says that if $\og_0$ does not disconnect
$-\infty$ from level zero and $G\setminus \og_0$ has infinitely many
connected levels, then $\og_0$ is \emph{nice} path.

Let us now address the obvious question of conditioning on a set of
measure zero in \eqref{eq:nice_A}. Let $\rho_0=\inf\{t> 0: S_t\in
\og_0\}$
and
\[h(z)=\Prob_1^z\{T_0<\rho_0\}.\]
Denote transition probabilities for the unconditioned random walk on
$G$ by $p(z,w)$. Then
\[h(z)=\sum_{w}p(z,w)h(w),\]
where we sum over all neighbors of $z$. In other words, $h(z)$ is
zero on the path and harmonic on the complement of the path. Suppose
$z$ and $w$ are on $D_j(\og_0)$ and $D_j(\og_0)$ is connected. Then
we have the following Harnack-type inequality:
\begin{equation} \label{harnack}
a\le \frac{h(z)}{h(w)}\le a^{-1}.
\end{equation}
where $a$ is the minimum over all possible connected configurations
$D_j$, over all $z$ and $w$ in $D_j$, of the probability that
starting at $z$ the random walk reaches $w$ before leaving $D_j$.
Since the torus has finitely many sites, $0<a<1/(2d)$. The constant
$a$ is repeatedly used in this article in estimates. It depends only
on the size of the torus, or, if generalized to a regular graph, on
the structure of the graph.

We start the random walk at $z\notin \og_0$ and we condition on the
random walk surviving up to level zero ($T_0<\rho_0$) to obtain a
process evolving according to the following transition probabilities
\begin{equation}\label{hprob}
\overline{p}(z,w)= \frac{\Prob_1\{S_1=w; T_0<\rho_0 |
  S_0=z\}}{\Prob_1\{T_0<\rho_0 | S_0=z\}}=\frac{p(z,w)h(w)}{h(z)}.
\end{equation}
Thus conditioning on $\{S[0,T_0]\cap \og_0=\emptyset\}$ simply means
that $S$ is an $h$-process conditioned to avoid $\og_0$, and
evolving according to transition probabilities given in
\eqref{hprob}. We can also define the hitting measure of level $-n$
by a random walk started at $|z|<-n$ and conditioned to avoid
$\og_0$ up to time $T_0$ as
\[\mu_{-n,z}(w)=\Prob_1^z\{S(T_{-n})=w; T_{-n}<\rho_0\}\frac{h(w)}{h(z)}.\]
Then we can show that if two $h$-processes conditioned on avoiding a
\emph{nice} path $\og_0$ are started far enough, then they can be
coupled by the time they hit a given level with high probability. In
fact, the reason for choosing this definition for \emph{nice} paths
was the need for such a coupling result. More general definitions of
\emph{nice} paths can be given, but we use this one in our present
work for simplicity. We prove the coupling result in the following
lemma.

\begin{lemma}\label{coupling0}
Let $n\ge 0$ be given. For every $\epsilon>0$, there exists an $m>n$
such that for all pairs $z,z'\in G_{-m}$, if $S$ and $S'$ are
$h$-processes started at $z$, and $z'$ respectively, with transition
probabilities as in \eqref{hprob}, we can define $S$ and $S'$ on the
same probability space $(\Omega_1,\mathcal{F}_1,\overline{\mu})$
such that
\[\overline{\mu}\{S(T_{-n})\neq S'(T'_{-n})\}<\epsilon/2.\]
Furthermore, $||\mu_{-n,z}-\mu_{-n,z'}||<\epsilon$, where
$\Vert\cdot\Vert$ denotes the total variation norm.
\end{lemma}

\begin{proof} Fix $n\ge0$ and let $\epsilon>0$ be given.
Let $T_j$ and $T_j'$ be the hitting time of level $j$ by
$h$-processes $S$ and $S'$ respectively. If $D_j$ is connected, then
the hitting measures on level $j+1$ for the two $h$-processes are
within a constant. More precisely, using \eqref{harnack}, one can
show
\[\frac{\mu_{j+1,z}(w)}{\mu_{j+1,z'}(w)}\ge a^2.\]
Then we can maximally couple $S$ and $S'$ on the same probability
space $(\Omega_1, \mathcal{F}_1,\overline{\mu})$, so that
\[\overline{\mu}\{S(T_{j+1})\neq
S'(T'_{j+1})\}=\frac{1}{2}||\mu_{j+1,z}-\mu_{j+1,z'}||\le\frac{1}{2}(1-a^2).\]
For a detailed discussion of coupling and a proof for existence of
the maximal coupling, we refer the reader to \cite{Coupling}. Once
$S$ and $S'$ are coupled, we run them together.

If $G_{-m,-n}\cap D(\og_0)$ has at least $k$ connected
cross-sections, then the two $h$-processes do not couple by the time
they reach level $-n$ only if they do not couple at any of the
connected cross-sections:
\[\overline{\mu}\{ S(T_{-n})\neq S'(T'_{-n})\}
\le\overline{\mu}\{ S(T_{j})\neq S'(T'_{j}) \mbox{  for all j} \in
[-m,-n)\} \le \left(\frac{1-a^2}{2}\right)^k.\] Moreover, from the
standard coupling inequality we obtain
\[||\mu_{-n,z}-\mu_{-n,z'}||\le 2\overline{\mu}\{S(T_{-n})\neq
S'(T'_{-n})\}\le 2\left(\frac{1-a^2}{2}\right)^k.\] Choose $k$ large
enough such that $(\frac{1-a^2}{2})^k\le \epsilon/2$. Then let $m$
be so that $G_{-m,-n}\cap D(\og_0)$ has at least $k$ connected
levels. Since $\og_0$'s complement has infinitely many connected
cross-sections, $m$ is finite.
\end {proof}

We can now show that given a \emph{nice} path $\og_0$, conditioning
on avoiding this path makes sense in \eqref{eq:nice_A}. We start
with the following lemma which basically says that given a
\emph{nice} path $\og_0$, we can define a hitting measure on level
$-n$ of the $h$-process induced by this conditioning.
\begin{lemma}
If $\og_0$ is \emph{nice}, then for each $n\ge 0$, there exists a
unique limiting measure
\[\mu_{-n}=\lim_{z\to -\infty} \mu_{-n,z}\]
\end{lemma}
\begin {proof}
Fix $n$. Let $\left(z_{k}\right)_{k=1}^{\infty}$ be a sequence in
$G_{-n}\cap D(\og_0)$, with the property
$\displaystyle{\lim_{k\to\infty}|z_k|=-\infty}$. Let $\epsilon>0$.
By Lemma \ref{coupling0}, there exists $m>0$ such that for all $z,
z' \in G_{-m}\cap D(\og_0)$,
 \[\displaystyle{\| \mu_{-n,z} - \mu_{-n,z'} \| <\epsilon}\]
Let $n_1$ be the smallest integer so that
$\left(z_{k}\right)_{k=n_1}^{\infty}$ is in $G_{-m}\cap D(\og_0)$.
Then for all $i$, $j \ge n_1$, $\| \mu_{-n,z_i} - \mu_{-n,z_j} \|
<\epsilon$ and so for every $w$ on level $-n$,
$\left\{\mu_{-n,z_k}(w)\right\}_{k=1}^{\infty}$ is a Cauchy sequence
converging to some $\mu_{-n}(w)$. Then clearly $\mu_{-n,z}
\Rightarrow \mu_{-n}$ as $z \to -\infty$.
\end{proof}

If one can define a probability measure on conditioned paths coming
from $-\infty$, then the random function $Z_n$ in \eqref{eq:nice_A}
is well defined. Let $\overline{\eta}_0$ be a half infinite path. We
define $\upsilon_{k}$ to be the measure on
$\overline{\eta}_0|_k=\eta_{-k+1}\eta_k\dots\eta_0$ , the
restriction of $\overline{\eta}_0$ to the last $k$ elements of the
path, in the following way:
\[\upsilon_{k}(\overline{\eta}_0|_k)=\mathcal{M}(\overline{\eta}_0|_k)1_{\{\overline{\eta}_0|_k\cap\og_0=\emptyset\}}\frac{\mu_{-k}(w_{-k})}{h(w_{-k})},\]
where $\mathcal{M}$ denotes the unconditioned random walk measure on
paths and $w_{-k}$ is the first site on level $-k$ reached by
$\overline{\eta}_0$. Note that $\{\upsilon_{k}\}_{k=1}^{\infty}$ is
a consistent sequence of measures, and hence by Kolmogorov Extension
Theorem, it can be extended to a measure on half-infinite paths
$\displaystyle{\upsilon:=\lim_{n\to \infty} \upsilon_{-n}}$, which
depends on the initial configuration $\og_0$. Thus, when we
condition on avoiding $\og_0$ up to level zero, we mean that the
measure induced by the $h$-process on half-infinite paths is given
by $\upsilon$.

\subsection{Exponent estimate} From the definition of the intersection
exponent, we know that $\displaystyle{q_n\approx
e^{-n\xi(\lambda)}}.$ However, for $\lambda$ restricted to a closed
interval, away from zero, we will show that $q_n$ and
$\overline{q}_n$ are within multiplicative constants of
$e^{-n\xi(\lambda)}$. Moreover, this will also hold for
$\E^{\og_0}[Z_n^{\lambda}]$ for all $\og_0\in \oA$. We fix
$\lambda_1>0$ and $\lambda_2<\infty$ and restrict $\lambda$ to
$[\lambda_1,\lambda_2]$.
\begin{proposition}\label{estconst}
For every $0<\lambda_1<\lambda_2<\infty$, there exist positive
constants $c_1$ and $c_2$ such that for all $n\ge 0$ and all
$\lambda\in [\lambda_1,\lambda_2]$,
\begin{equation}\label{asymp}
c_1 e^{-\xi(\lambda)n}\le q_n\le c_2 e^{-\xi(\lambda)n}.
\end{equation}
Note that $c_1$ and $c_2$ can be chosen so they are independent of
$\lambda\in [\lambda_1,\lambda_2]$.
\end{proposition}

We will use the notation $q_n \asymp e^{-\xi(\lambda)  n}$, to mean
that $q_n$ is bounded as in \eqref{asymp}. The reason to restrict
$\lambda$ between two values $\lambda_1$, $\lambda_2$ is to get
constants uniform in $\lambda$. We proceed to prove the proposition,
but we will first need a couple of estimates and technical lemmas.

\subsubsection{Preparation lemmas}
We define the following stopping times:
\begin{itemize}
\item for $j\le 0$, let $\eta_{j}=\min\{t>T_0: |S_t|=j\}$ and
$\eta^1_j=\min\{t>T_0^1: |S_t^1|=j\}$
\item for $j>0$, let
$\eta_{j}=\min\{t>T_j: |S_t|=j\}$ and $\eta^1_j=\min\{t>T_j^1:
|S_t^1|=j\}$ \end{itemize} In other words, for non-positive $j$,
$\eta_{j}$ is the first time after reaching level zero that $S_t$
returns to level $j$ and, for positive $j$,  $\eta_{j}$ is the
second time $S_t$ reaches level $j$.

\begin{lemma}
For any $z$ with $|z|=0$,
\begin{equation}\label {ARWestimate}
\Prob_1^z\{T_n<\eta_0\}\ge \frac{2p-1}{d}.
\end{equation}
\end{lemma}

\begin{proof} Let
$\displaystyle{\psi(x)=[(1-p)/p]^x}$. If $a<0<b$, one can easily
solve
\[\Prob_1\{T_{b}<T_{a}\}=\frac{\psi(0)-\psi(a)}{\psi(b)-\psi(a)}.\]
See \cite{Durrett} for a standard proof of such a result. The random
walk started on level zero will reach level $n$ before returning to
level zero if and only if it first goes one step forward and then
from level $1$ it reaches level $n$ before reaching level $0$. But,
because the state space is a cylinder, starting from $1$ and
reaching $n$ before $0$ is equivalent to starting at $0$ and
reaching $n-1$ before $-1$. Hence we can use the above expression
with $a=-1$ and $b=n-1$:
\begin{equation*}
\Prob_1^z\{T_n<\eta_0\}=\frac{p}{d}\cdot\Prob_1^{|z'|=1}\{T_{n}<T_{0}\}=\frac{p}{d}\cdot
\frac{\psi(0)-\psi(-1)}{\psi(n-1)-\psi(-1)}.
\end{equation*}
Now plugging in $\psi$, and recalling that $p>1/2$ and hence
$p>1-p$, we obtain
\[\frac{\psi(0)-\psi(-1)}{\psi(n-1)-\psi(-1)}=\frac{2p-1}{p}\cdot\frac{1}{1-[(1-p)/p]^n}\ge\frac{2p-1}{p},\]
and the lemma follows immediately. Similarly,
$\displaystyle{\Prob\{T_n^1<\eta_0^1\}\ge \frac{2p-1}{d}.}$
\end{proof}

\begin{lemma}\label{lemA}
For all $k$, all $\og_1 \in \mathcal{A}_1$, and all $z$ on level
zero of $G$,
\begin{equation}\label{estk}
\Prob_1^z\{S[0,T_1]\cap \og_1 =\emptyset| S[0,T_1]\cap
G_{-k}\neq\emptyset \}\le \frac{d}{p}\, \Prob_1^z\{S[0,T_1]\cap
\og_1 =\emptyset\}
\end{equation}
\end{lemma}

\begin{proof} First, on the event
$\{z \notin \gamma_1\}$, since we can avoid both $\og_0$ and
$\gamma_1$ by simply taking one step in the $\Z$ direction of $G$,
  we have $\Prob^z_1\{S[0,T_1]\cap \og_1 =\emptyset\}\ge p/d$ and so,
\[\Prob_1^z\{S[0,T_1]\cap \og_1 =\emptyset| S[0,T_1]\cap
  G_{-k}\neq\emptyset\}\le \frac{d}{p}\cdot \frac{p}{d} \le \frac{d}{p}\, \Prob_1^z\{S[0,T_1]\cap \og_1 =\emptyset\}.\]
Secondly, on the event $\{z \in \gamma_1\}$, we also get
\[\Prob_1^z\{S[0,T_1]\cap \og_1 =\emptyset| S[0,T_1]\cap G_{-k}\neq\emptyset\}=0= \frac{d}{p}\, \Prob_1^z\{S[0,T_1]\cap \og_1 =\emptyset\}\]
Furthermore, if $A$ is any set in the $\sigma$-algebra generated by
$S$ up to time $T_1$, conditioning on this event $A$ instead of
$\{S[0,T_1]\cap G_{-k}\neq\emptyset\}$, yields the same inequality
as \eqref{estk}.
\end{proof}

From the proof of Lemma \ref{lemA}, also note that if $z$ and $z'$
are on level zero of $G$ and $z, z'\notin\gamma_1$,
\begin{equation}
\Prob_1^z\{S[0,T_1]\cap \og_1 =\emptyset\} \le \frac{d}{p}\,
\Prob_1^{z'}\{S[0,T_1]\cap \og_1 =\emptyset\}
\end{equation}
In order to show $q_n\asymp e^{-\xi(\lambda) n}$, we need some
preliminary results which we summarize in the two lemmas below.
First we consider the random walk started at $z$, $|z|=0$, that
reaches level $n$ without hitting $\og_n$, while staying above level
$-1$. We define random variables
\[\hat{Z}_{n,z}=
\Prob_1^z\{S[0,T_n]\cap \og_n =\emptyset; T_n<\eta_{-1}\},\]
\[\displaystyle{\hat{Z}_n=\sup_{|z|=0}\hat{Z}_{n,z}}.\]

\begin{lemma}\label{separation1}
There exists a constant $c$ such that for all $n$, $\hat{Z}_n\ge
c\overline{Z}_n$. In particular,
\[\displaystyle{\sup_{\og_0}\E^{\og_0}[\hat{Z}_{n}^{\lambda}]\ge c \overline{q}_n}.\]
\end{lemma}

\begin{proof}
Suppose we start $S$ at $z$, with $|z|=0$. We prove the lemma by
first looking at the random walk on the event it falls below level
zero before time $T_n$. Let
\[Y_{n,z}=\Prob_1^z\{S[0,T_n]\cap\og_n=\emptyset; T_n>\eta_{-1}\}.\]
Let $A_n$ be the event that the random walk $S$, up to time $T_n$,
hits every site on level zero. Since the cross-section of the
cylinder has a finite state space, for all $n$, $\Prob_1(A_n)\ge a$,
for some same constant $a>0$ that can be easily computed. Similarly,
given $T_n>\eta_{-1}$, let $A_{\eta_{-1}}$ be the event that the
random walk $S$, up to time $\eta_{-1}$, hits every site on level
zero, with $\Prob_1(A_{\eta_{-1}}|T_n>\eta_{-1})\ge a$. Here
$A_{\eta_{-1}}$ is a subset of $\{S[0,\eta_{-1}]\cap \og_n\neq
  \emptyset\}$ which implies
\begin{eqnarray*}
Y_{n,z}&\le& \Prob_1^z\{S[0,\eta_{-1}]\cap \og_n=\emptyset|
T_n>\eta_{-1}\}
\Prob_1^z\{T_n>\eta_{-1}\}\\
&&\quad \quad\times \Prob_1^z\{S[0,T_n]\cap \og_n=\emptyset
|S[0,\eta_{-1}]\cap \og_n=\emptyset;T_n>\eta_{-1}\}\\
&\le&(1-a)\sup_{|z'|=-1}\Prob_1^{z'}\{S[0,T_n]\cap
\og_n=\emptyset\}\le(1-a)\overline{Z}_{n}
\end{eqnarray*}
Taking the supremum over all $z$ yields
$\displaystyle{Y_n=\sup_{|z|=0}Y_{n,z}\le (1-a)\, \overline{Z}_n.}$
Since $\overline{Z}_n\le \hat{Z}_n+ Y_n$, this implies $\hat{Z}_n\ge
a\overline{Z}_n$ which proves the lemma, with a constant
$c=a^{\lambda_2}$, uniform in $\lambda$.
\end{proof}

\begin {lemma}\label{separation2}
There exists a constant $c$ such that for all $n$, all histories
$\og_m$, and all $z \notin \og_m$ with
  $|z|=m$,
\[\E^{\og_m}[\hat{Z}_{n+m,z}^{\lambda}1_{
\{T^1_{n+m}<\eta^1_m\}}]\ge c\, \overline{q}_n.\]
\end{lemma}
\begin{proof} Without loss of generality, assume $m=0$. Recall that
$S$ is started on level zero of $G$. Let $x_0$ be the endpoint of
$\og_0$, where we attach $\hg_n$. Consider another starting
configuration $\ogp_0$ with endpoint at $x'_0$ to which we attach
$\hg_n$ (translated accordingly in the $\T$ direction of $G$). We
let $\og_n=\og_0\hg_n$ and $\ogp_n=\ogp_0\hg_n$. First note that if
$x_0=x'_0$, and if the random walk S does not fall below level zero,
$S$ can only hit $\hg_n$, and so
$\hat{Z}_{n,z}(\og_n)=\hat{Z}_{n,z}(\ogp_n)$. For the case when
$\og_0$ and $\ogp_0$ do not have the same endpoint, we start one
random walk at $z$ and consider $\hat{Z}_{n,z}(\og_n)$. We can find
a point $z'$ on level zero of $G$ such that the ``relative
position'' of $z'$ to $x'_0$ is the same as the ``relative
position'' of $z$ to $x_0$. And again, if $S$ does not fall below
level zero, we get
\begin{equation}\label{hat1}
\hat{Z}_{n,z}(\og_n)=\hat{Z}_{n,z'}(\ogp_n)
\end{equation}
Now, starting a random walk anywhere on level
  zero, away from $\og_0$, say at $z$, we claim that for all
  $z'$ on level zero, $z'\notin \og_0$,
\begin{equation}\label{zz'}
\overline{Z}_{n,z}1_{\{T_n^1<\eta^1_0\}}\ge
a\overline{Z}_{n,z'}1_{\{T_n^1<\eta^1_0\}}
\end{equation}
This is easy to see: starting the random walk at $z$, with
probability greater than $a$ it will reach $z'$ before hitting level
$1$, or returning to level $-1$, without hitting the starting point
of $\hg_n$. Note that in this case, the random walk from $z$ to $z'$
will not hit $\hg_n$ and the claim follows. Furthermore, by the same
argument,
\begin{equation}\label{hat2}
\hat{Z}_{n,z}1_{\{T_n^1<\eta^1_0\}}\ge
a\hat{Z}_{n,z'}1_{\{T_n^1<\eta^1_0\}}
\end{equation}
Since \eqref{hat2} holds for all $z'$ not on $\og_0$, taking the
supremum over all $z'$, along with equation \eqref{hat1},
\[\hat{Z}_{n,z}(\og_n)1_{\{T_n^1<\eta^1_0\}}\ge a\hat{Z}_{n}(\ogp_n)1_{\{T_n^1<\eta^1_0\}},\]
so averaging over all $\hg_n$, we get
\[\E^{\og_0}[\hat{Z}^{\lambda}_{n,z}1_{\{T_n^1<\eta^1_0\}}]\ge a^{\lambda_2}\E^{\ogp_0}[\hat{Z}_{n}^{\lambda}1_{\{T_n^1<\eta^1_0\}}],\]
  Observe that this inequality holds for any pair $\og_0$, $\ogp_0$,
  and thus,
\[\E^{\og_0}[\hat{Z}^{\lambda}_{n,z}1_{\{T_n^1<\eta^1_0\}}]\ge
  a^{\lambda_2}\sup_{\ogp_0}\E^{\ogp_0}[\hat{Z}^{\lambda}_{n}1_{\{T_n^1<\eta^1_0\}}].\]
Next we want to show there exists a constant $c'<1$ such that
\begin{equation}\label{complement}
\sup_{\og_0}\E^{\og_0}[\hat{Z}^{\lambda}_{n}1_{\{
T_n^1>\eta^1_0\}}]\le
c'\,\sup_{\og_0}\E^{\og_0}[\hat{Z}^{\lambda}_{n}].
\end{equation}
This, along with Lemma \ref{separation1}, will finish the proof with
a constant equal to $(1-c')a^{2\lambda_2}$. In order to prove
(\ref{complement}), we condition on the event that $\hg_n$ returns
to level zero before reaching level $n$. Then a random walk started
on level zero will avoid $\og_n$ only if it avoids the part of
$\hg_n$ from the first return to level zero up to the hitting time
of level $n$. If we denote this part of $\hg_n$ by $
S^1[\eta^1_0,T_n^1],$
\[\E^{\og_0}[\hat{Z}_{n}^{\lambda} | T_n^1>\eta^1_0]\le
\E^{\og_0}[\sup_z\Prob_1^z\{S[0,T_n]\cap
  S^1[\eta^1_0,T_n^1]=\emptyset; T_n<\eta_{-1}\}^{\lambda}]\le
\sup_{\og_0}\E^{\og_0}[\hat{Z}_{n}^{\lambda}]\] where the second
inequality follows from observing that $\hat{Z}_{n,z}$
  depends on $\og_0$ only in terms of the endpoint of
  $\og_0$. Recall that $\Prob\{T_n^1>\eta^1_0\}\le 1-(2p-1)/d$, hence taking
  $c'=1-(2p-1)/d>0\,$ completes our proof.

Furthermore, for all $n$ and  for all $\og_0$, we have
 $\E^{\og_0}[\overline{Z}_{n}^{\lambda}1_{\{
  T_n^1<\eta^1_0\}}]\ge c\overline{q}_n$.
\end{proof}

\subsubsection{Proof of Proposition \ref{estconst}}
We prove the proposition by showing that for all $n\ge 1$, both $q_n
\asymp e^{-\xi(\lambda) n}$ and $\overline{q}_n \asymp e^{-\xi
(\lambda)n}$.

By subadditivity of $\log(q_n)$, we have
\[\displaystyle{\lim_{n\to \infty}\frac{\log(q_n)}{n}=\inf_n
  \frac{\log(q_n)}{n}=-\xi(\lambda)}\]
and thus, for all $n$,
\begin{equation}\label{2}
q_n\ge e^{-\xi(\lambda) n},
\end{equation}
Note that this means $c_1=1$.

We claim there exists a constant $\hat{c}$ such that for all $n$,
$q_n\ge \hat{c}\overline{q}_n$. Then, along with the trivial
inequality $q_n\le\overline{q}_n$, this implies
\begin{equation}\label{1}
q_n\asymp\overline{q}_n.
\end{equation}
On the event where $\og_n$ does not reach level zero
  after time $0$, the random walk $S$, coming from $-\infty$ and conditioned to
  avoid $\og_0$, will avoid $\og_n$ if and only if $S[T_0,T_n]$ avoids
  $\og_n$. This is the same as starting a random walk at $S(T_0)$ and
  avoiding $\og_n$. But, from (\ref{zz'}) we know this is bounded
  below by $a\overline{Z}_{n,z}$, for all $z\notin \og_0$. Taking the supremum over all $z$, we
  have
  $Z_n 1_{\{T_n^1<\eta^1_0\}}
\ge a\overline{Z}_n 1_{\{T_n^1<\eta^1_0\}}$. From Lemma
\ref{separation2},
\[\E^{\og_0}[Z_n^{\lambda}1_{\{T_n^1<\eta^1_0\}}]\ge
a^{\lambda}\E^{\og_0}[\overline{Z}_{n}^{\lambda}1_{\{T_n^1<\eta^1_0\}}]\ge
a^{\lambda}\, c \overline{q}_n .\] We let $\hat{c}=a^{\lambda_2}\,
c=a^{3\lambda_2}(2p-1)/d$ and then $q_n\ge \hat{c}\overline{q}_n$,
with a constant uniform in $\lambda$. Furthermore,
\begin{equation}\label{chat}
\E^{\og_0}[Z_n^{\lambda}1_{\{T_n^1<\eta^1_0\}}]\ge
\hat{c}\E^{\og_0}[Z_n^{\lambda}].
\end{equation}

In the last part of the proof we will show there exists a constant
$c_2$, uniform in $\lambda$ such that
\begin{equation}\label{3}
\overline{q}_n\le c_2 e^{-\xi(\lambda) n}
\end{equation}
Then the proposition follows immediately from inequalities
(\ref{1}), (\ref{2}) and (\ref{3}). To prove inequality (\ref{3}),
we bound $\overline{Z}_{n+m,z}$ by the probability $S$ avoids
$\og_{n+m}$ while not going below level $n$ between times $T_n$ and
$T_{n+m}$. We are considering this probability only on the event
where $S^1$ reaches level $n+m$ before returning to level $n$.
Intuitively, we want $S$ and $S^1$ to have a nice behavior from
level $n$ onward, so we can ``separate" what happens up to level $n$
from what happens from level $n$ to $m+n$. Then for every pair $n,m$
we have the following relation between $\overline{q}_{n+m}$,
$\overline{q}_n$ and $\overline{q}_m$:

\begin{equation*}
\begin{split}
\overline{q}_{n+m}
&= \sup_{\og_0}\E^{\og_0}[\sup_{|z|=0}\Prob_1^z \{S[0, T_{m+n}]\cap \og_{m+n} =\emptyset\}^{\lambda}]\\
&\ge \sup_{\og_0}\E^{\og_0}[\sup_{|z|=0}\Prob_1^z \{S[0, T_{m+n}]\cap \og_{m+n} =\emptyset; S(T_n,T_{n+m}]\cap G_{n-1}=\emptyset\}^{\lambda}1_{\{T_{n+m}^1<\eta^1_n\}}]\\
&\ge \sup_{\og_0}\E^{\og_0}[\sup_{|z|=0}\overline{Z}_{n,z}^{\lambda}
\E^{\og_0}[\Prob_1 \{S[0, T_{m+n}]\cap \og_{m+n} =\emptyset;\\
&\qquad\qquad\qquad\qquad
  S(T_n,T_{n+m}]\cap G_{n-1}=\emptyset |S[0, T_n]\cap \og_n
=\emptyset\}^{\lambda}1_{\{T_{n+m}^1<\eta^1_n\}}|\mathcal{F}_n]]\\
&= \sup_{\og_0}\E^{\og_0}[\sup_{|z|=0}\overline{Z}_{n,z}^{\lambda}
\E^{\og_n}[\hat{Z}_{n+m,S(T_n)}^{\lambda}
1_{\{T_{n+m}^1<\eta^1_n\}}]]\\
&\ge
\sup_{\og_0}\E^{\og_0}[\sup_{|z|=0}\overline{Z}_{n,z}^{\lambda}(\,c\overline{q}_m)]\\
&\ge c \overline{q}_m\, \overline{q}_n
\end{split}
\end{equation*}
with the second to last step following from Lemma \ref{separation2}.
Note that $\log(c\, \overline{q}_n)$ is a super-additive function,
and then using (\ref{1})
\[\displaystyle{\sup_n \frac{\log(c\, \overline{q}_n)}{n}=\lim_{n\to \infty}\frac{\log(c\, \overline{q}_n)}{n}=\lim_{n\to \infty}\frac{\log(q_n)}{n}=-\xi(\lambda)}\]
and so, for all $n$, $\overline{q}_n\le c_2 e^{-\xi (\lambda)n}$,
where $c_2=\left[a^{2\lambda_2}(2p-1)/d\right]^{-1}$ and the proof
is complete.

From the proof to Proposition \ref{estconst}, we also get the
following result:
\begin{corollary}\label{estconst2}
For all $n$ and all $\og_0\in \oA$, $\E^{\og_0}[Z_n^{\lambda}]\asymp
e^{-\xi (\lambda)n}$.
\end{corollary}

\section{Exponential convergence of Markov chains}\label{sec:Exp Convergence}

In this section we will first define a Markov process on pairs of
non-disconnecting random walk paths on $G$. Each
$\xi(\lambda):=\xi(1,\lambda)$ will then be associated to such a
Markov process, along with a weighting that corresponds to the value
of $\lambda$. We will fix $\lambda$ and start two Markov chains $X$
and $X'$, with different initial configurations $\og_0$ and $\ogp_0$
respectively. The goal is to show that the two chains can be coupled
at exponential rate, and as a result we will be able to describe an
invariant limiting measure on half-infinite paths.

\subsection{Markov process on random walk paths}\label{MC}

Fix $N$ large and start with an initial configuration $\og_0$ from
$\oA$. Attaching a random walk started at the endpoint of $\og_0$
and run until it hits level $N$ gives a path $\og_N$, which,
translated accordingly, produces $\tg_N$. If $\tg_N$ does not
disconnect $-\infty$ from level zero, then it is a \emph{nice} path.
(This follows from the fact that $\hg_N$ is finite almost surely and
hence one can find a level $m$ such that $\hg_N$ does not fall below
level $m$; but $G_{m-n}\cap D(\tg_n)=G_{m}\cap D(\og_0)$ must have
infinitely many connected levels, hence $\tg_N$ is \emph{nice}.)
Each $\og_N$ is weighted by $\displaystyle{e^{-\lambda \Phi_N}}$ and
then normalized by $\E^{\og_0}[e^{-\lambda \Phi_N}]$ to get a
probability measure $\Q_N=\Q_N^{\og_0}$. Clearly, paths $\og_N$ that
cannot be avoided by an $h$-process from $-\infty$ to level $N$ will
be assigned measure zero and we will say that these paths "do not
survive" up to level $N$. For a given $\og_0$, we denote the
expectation with respect to the measure $\Q^{\og_0}$ by
$\E_{\Q}^{\og_0}$. Let
\[K_N(\og_0)=e^{\xi (\lambda)N}\E^{\og_0}[e^{-\lambda \Phi_N}].\]
From Corollary \ref{estconst2}, we get $K_N(\og_0)\asymp 1$, i.e.
$K_N(\og_0)$ are bounded above and below by positive constants,
uniform in $\og_0$ and $N$. Furthermore, we can show that
$\displaystyle{\lim_{N\to\infty}K_N(\og_0)}$ exists. A proof of this
result can be found in Section \ref{sec:analyticity}, as a part of
the proof of analyticity of $\xi(\lambda).$

If $m<N$, we have
\[\E^{\og_0}[e^{-\lambda \Phi_N}|\mathcal{F}_m]=
e^{-\lambda \Phi_m}\E^{\tg_m}[e^{-\lambda \Phi_{N-m}}].\] Let
$\mathcal{M}$ be the unconditioned random walk measure on paths. We
let the path $\og_n$ evolve, starting with $\og_0$ and conditioned
to survive up to level $N$. Then, for this conditioned random walk
we write transition probabilities:
\begin{equation}\label{transitions}
Q_N(\tg_1|\tg_0)=e^{-\lambda
\Phi(\tg_1)}e^{\xi(\lambda)}\frac{K_{N-1}(\tg_1)}{K_N(\tg_0)}\mathcal{M}(\hg_1),
\end{equation}
The $m$-step transition probabilities, for $m<N$ are given by
\[
Q_N(\tg_m|\tg_0)=e^{-\lambda \Phi_m}e^{\xi(\lambda)
m}\frac{K_{N-m}(\tg_m)}{K_N(\tg_0)}\mathcal{M}(\hg_m).\]

Here we have used the fact that $\og_0=\tg_0$ and we will use the
two notations interchangeably. It is trivial to check that if
$n+m\le N$, the distribution of $\tg_{m+n}$ under $\Q_N^{\tg_0}$ is
the same as the distribution of $\tg_{m+n}$ under
$\Q_{N-n}^{\tg_n}$. Intuitively, this should be the case since we
condition on paths surviving up to level $N$ and once they have
reached level $n$ they only have to survive another $N-n$ steps.
Then transition probabilities from \eqref{transitions} describe a
time-inhomogeneous Markov chain $X_n$ taking values in $\oA$ and
conditioned to survive for a total of $N$ levels, with $X_0=\og_0$
and $X_n:=\tg_n$ for $n\ge 1$.

Starting with a different \emph{nice} initial configuration
$\ogp_0$, we construct a Markov chain $X'_n$ with transition
probabilities given by a formula similar to \eqref{transitions} and
$X'_0=\ogp_0$. Then we will show that we can couple $X_n$ and $X'_n$
as in Theorem \ref{expc}. Let $\eta|_k$ denote the restriction to
the last $k$ levels of a half-infinite path $\eta$ from $\oA.$ Then
Theorem \ref{expc} implies the following corollary.
\begin{corollary} \label{coupling}
For all $\og_0, \ogp_0, \eta$ in $\oA$ and all $0\le k\le N/2$,
\[\sum_{\eta|_k}\left|\frac{\E^{\og_0}[e^{-\lambda \Phi_N};
\tg_N=_k\eta]}{\E^{\og_0}[e^{-\lambda \Phi_N}]}-
\frac{\E^{\ogp_0}[e^{-\lambda \Phi_N};
\tg'_N=_k\eta]}{\E^{\ogp_0}[e^{-\lambda\Phi_N}]}\right|
=O(e^{-\beta_1 N}).\]
\end{corollary}

\begin{proof} Let $X_n$ and $X'_n$ be Markov chains given by initial
  configuration $\og_0$ and $\ogp_0$, respectively, and evolving
  according to transition probabilities given in
  \eqref{transitions}. Observe that
\[\frac{\E^{\og_0}[e^{-\lambda \Phi_N};
\tg_N=_k\eta]}{\E^{\og_0}[e^{-\lambda
\Phi_N}]}=\Q_N^{\og_0}\{X_{N}=_k \eta\}.\] Then for all $k \le N/2$,
we have
\begin{eqnarray*}
\sum_{\eta|_k}\left\vert\Q_N^{\og_0}\{X_{N}=_k \eta\}-
\Q_N^{\ogp_0}\{X'_N=_k \eta\}\right\vert &\le&
2\tProb\{X_N\not\equiv_k X'_N\},
\end{eqnarray*}
and the corollary follows from noting that for all $k\le N/2$,
\[\displaystyle{\tProb\{X_{N}\not\equiv_k X'_{N}\}\le
\tProb\{X_{N}\not\equiv_{N/2}X'_{N}\}\le Ce^{-\beta_1 N}.}\] This
result implies existence of an invariant measure on paths $\og_0\in
\mathcal{A}$ which we will prove in Section \ref{sec:Invariant
Measure}.

\end{proof}

The rest of Section \ref{sec:Exp Convergence} is devoted to proving
Theorem \ref{expc}. First we present some technical lemmas, followed
by the description of the coupling in Section \ref{sec:coup} and the
proof of the theorem.

\subsection{Preliminary estimates}
For all $i\le 0$, and $\tg \in \oA$, let $J_i$ be the random
variable that takes the value 1 if the $i$-th cross-section of
$G\setminus \tg$ is connected, and is zero otherwise. Then for $k\ge
j\ge 0$, the sum $\sum_{i=-k}^{-j}J_i $ represents the number of
connected cross-sections between levels $-k$ and $-j$ in $G\setminus
\tg$. For all $k\ge0$ and $j>0$, we define the sets
\[V_{k,j}(\delta)=\left\{\tg\in \oA: \sum_{i=-k}^{-k+j-1}J_i>\delta j\right\}\]
\[W_{k,j}=\left\{\gamma\in \mathcal{X}_k: \gamma \cap G_{j}=\emptyset\right\}.\]
Then $V_{k,j}(\delta)$ is the set of all $\tg$ in $\oA$ with the
property that it has at least $\delta j$ connected cross-sections
between levels $-k$ and $-k+j-1$, and $W_{k,j}$ is the set of all
$\gamma$ started on level $k-1$ and reaching level $k$ before
reaching level $j$.

The set $V_{k,j}$ is large for $j$ large enough, that is, for an
appropriate $\delta$, the $\Q_N$-probability that $\tg_k$ is not in
$V_{k,j}(\delta)$ decays exponentially in $j$, and this is
independent of $\tg_0$. Therefore, the weighted measure put on paths
that are not in $V_{k,k/2}$ is exponentially small. Moreover, we
will show that the set $V_k$ of all paths that have enough connected
cross-sections between levels $-k$ and $-k+j-1$ for all integers $j$
in $[k/2,k]$, is also large. Note that
\[V_k=\{\tg\in \oA: \tg\in \bigcap_{j=k/2}^k V_{k,j}\},\]
and the set $V_k$ is a subset of $V_{k,k/2}$.

\begin{lemma}\label{largedev}
There exist constants $\alpha'>0$ and $\delta>0$ such that for all
$\og_0\in\mathcal{A}$, all $n\le N$ and $k\le n$,
\[\Q_N^{\og_0}\{\tg_k\notin V_k(\delta)\}\le c\,e^{-\alpha' k/2}.\]
\end{lemma}

\begin{proof} Using Chebyshev's inequality, for all integers $j \in [k/2,k]$
\begin{eqnarray*}
\Q_N^{\og_0}\{\tg_k\notin V_{k,j}(\delta)\}
&=&\Q_N^{\og_0}\left\{\exp\{-t\sum_{i=-k}^{-k+j-1}J_i\}
\ge e^{-t\delta j}\right\} \\
&\le& e^{t\delta
j}\E^{\og_0}_{\Q}\left[\exp\{-t\sum_{i=-k}^{-k+j-1}J_i\}\right].
\end{eqnarray*}
Suppose we know the path $\tg_k$ up to level $i-1$ (note that $i<0$
and we have information from $\mathcal{F}_{k+i-1}$). What is the
$\Q_N^{\tg_0}$-probability that the chain evolved so that level
$i-1$ is connected? If, starting from level $i-1$, the path moved
forward in the $\Z$ direction of $G$ and it did not return to this
level, then level $i-1$ remained connected. It follows from
$\eqref{chat}$ that
$\displaystyle{\Q_N^{\og_0}\{J_i=1|\mathcal{F}_{k+i-1}\}\ge
\hat{c}}$, and so
\begin{eqnarray*}
\E_{\Q}^{\og_0}\left[\exp\{-t\cdot J_i\} \big|\,
\mathcal{F}_{k+i-1}\right] &=&
e^{-t}\Q_N^{\og_0}\{J_i=1|\mathcal{F}_{k+i-1}\}+\Q^{\og_0}\{J_i=0|\mathcal{F}_{k+i-1}\}\\
&\le& e^{-t}+1-\hat{c}.
\end{eqnarray*}
Using this, one can check that
\begin{eqnarray*}
\E_{\Q}^{\og_0}\left[\exp\{-t\sum_{i=-k}^{-k+j-1}J_i\}\right] &\le&
(e^{-t}+1-\hat{c})^j.
\end{eqnarray*}
Therefore, $\displaystyle{\Q_N^{\og_0}\{\tg_k\notin
V_{k,j}(\delta)\} \le e^{t\delta j}(e^{-t}+1-\hat{c})^j}$. Fix $t$
large enough such that $e^{-t}+1-\hat{c}<1$, then let
$\alpha'=-\delta t-\log(e^{-t}+1-\hat{c})$, and choose $\delta\le
1/4$, so that $\alpha'>0$. We now let $V_{k,j}=V_{k,j}(\delta)$ for
this particular choice of $\delta$.

To get an estimate on the size of $V_k$, we need only consider
$j\in[k/2,k]$. A path $\tg_k$ is not in $V_k$ if it is not in at
least one of the sets $V_{k,j}$ for $j$ between $k/2$ and $k$:
\[\Q_N^{\og_0}\{\tg_k\notin V_k\}\le \sum_{j=k/2}^k\Q_N^{\og_0}\{\tg_k\notin V_{k,j}\}
\le \frac{e^{-\alpha' k/2}}{1-e^{-\alpha'}},\] and the lemma
follows.
\end{proof}
We start with $\tg_0=_k\tg'_0$ satisfying $\tg_0\in V_k$ (this
implies $\tg'_0$ also in $ V_k$). To $\tg_0$ and $\tg'_0$ we attach
the same $\hg_1\in W_{1,-k/2}$. We want to show
$e^{-\lambda\Phi(\tg_1)}$ is close to $e^{-\lambda\Phi'(\tg'_1)}$,
and $K_N(\tg_0)$ is close to $K_N(\tg'_0)$. The estimate on how
close these quantities are relies on a coupling of $h$-processes.
Let $S$ and $S'$ be $h$-processes given by random walks conditioned
to avoid $\og_0$, and $\ogp_0$ respectively. If $\og_0=_k\ogp_0 \in
V_k$, then one can show the $h$-processes $S$ and $S'$ can be
coupled by the time they first hit level $-k/2$ with high
probability.

\begin{lemma}\label{coupling1} There exist constants $c,\alpha''>0$
  such that for all $k\ge 0$, for all $\og_0, \ogp_0 \in
V_k$, with $\og_0=_k\ogp_0$, if $S$ and $S'$ are $h$-processes
described as above, then $S$ and $S'$ can be defined on the same
probability space $(\Omega_1, \mathcal{F}_1,\overline{\mu})$ such
that
\[\overline{\mu}\{S(T_{-k/2})\neq S'(T'_{-k/2})\}\le
c\,e^{-\alpha'' k}.\]
\end{lemma}

\begin {proof}
For $-k<j< -k/2$, let $\mu_{j}(x)$ be the hitting measure on level
$j$ of the $h$-process $S$ conditioned to avoid the path $\og_0$,
and $\mu'_{j}(x)$ be the hitting measure on level $j$ of the
$h$-process $S'$ conditioned to avoid the path $\ogp_0$. Then, if
level $j$ of $\og_0$ is connected, we claim there exists a constant
$c>0$ such that
\begin{equation}\label{eq:coupling1}
\frac{\mu_{j+1}(w)}{\mu'_{j+1}(w)}\ge c.
\end{equation}
Let $h(x)=\Prob_1^x\{S[0,T_0]\cap \og_0=\emptyset\}$ and
$h'(x)=\Prob_1^x\{S'[0,T'_0]\cap \ogp_0=\emptyset\}$. Using
\eqref{harnack}, for all $ x, x'\notin \og_0$, with $|x|=|x'|=j$, we
have
\[\frac{\mu_{j+1,x}(w)}{\mu'_{j+1,x'}(w)}\ge
\left(\frac{p}{d} a\right) \frac{h(w)}{h(x)}\frac{h'(x')}{h'(w)}\ge
\left(\frac{p}{d} a^3\right) \frac{h(w)}{h(z)}\frac{h'(z)}{h'(w)},\]
where $z=w-(1,0)$ is the last point on level $j$ touched by $S$ and
$S'$ before reaching level $j+1$ at $w$. Furthermore,
$\displaystyle{\frac{1-p}{d}h(z)\le h(w)\le
  \frac{d}{p}h(z)}$ and a similar inequality holds for $h'(w)$.
  Thus, for all $ x,x' \notin \og_0$,
with $|x|=|x'|=j$,
\[\frac{\mu_{j+1,x}(w)}{\mu'_{j+1,x'}(w)}\ge \frac{p^2(1-p)}{d^3}a^3>0. \]
Then inequality \eqref{eq:coupling1} follows from noting that
$\displaystyle{\mu_{j+1}(w)=\sum_{|x|=j}\mu_{j}(x)\mu_{j+1,x}(w).}$

Recall that on $V_k$, both $D(\og_0)$ and $D(\og'_0)$ have $k\delta$
connected cross-sections between levels $-k$ and $-k/2$. Using the
same coupling as in Lemma \ref{coupling0}, we can define $S$ and
$S'$ on the same probability space
$(\Omega_1,\mathcal{F}_1,\overline{\mu})$, with
\[\overline{\mu}\{S(T_{-k/2})\neq S'(T'_{-k/2})\}\le
2\left(\frac{1-c}{2}\right)^{\delta k},\] and then let
$\alpha''=-\delta\log\left(\frac{1-c}{2}\right)$ to complete the
proof.
\end{proof}

We are now ready to estimate $K_N(\tg_0)/K_N(\tg'_0)$.
\begin{proposition}\label{cnestimate}
There exists $\beta>0$ such that for all $n\le N$, all $k\ge 0$, and
all histories $\og_0,\ogp_0$ in $V_k$ with $\og_0=_k\ogp_0$,
\[K_n(\og_0)=K_n(\ogp_0)[1+O(e^{-\beta k})].\]
\end{proposition}

\begin{proof} To simplify notation, let
\[Z_n^*(\tg_n)=\Prob_1\{S(-\infty,T_n]\cap
\og_n=\emptyset | S(-\infty,T_0]\cap\og_0=\emptyset; S[T_0,T_n]\cap
G_{-k}\neq\emptyset\}.\] Let $\U$ be the event $\{\hg_n\cap
G_{-k/2}=\emptyset\}$. Then on $\U$, using the coupling result from
Lemma \ref{coupling1}, and an estimate similar to
\eqref{ARWestimate},
\begin{equation}\label{bound}
Z_n(\tg_n)\le Z_n(\tg'_n)+\left(\frac{1-p}{p}\right)^k Z_n^*(\tg_n)+
\left(c\,e^{-\alpha'' k}\right)Z_n(\tg_n)
\end{equation}
Using Lemma \ref{separation2}, it is easy to show that
\begin{equation*}
\E^{\og_0}[(Z^*_n)^{\lambda}1_{\U}]\le
c\E^{\og_0}[Z_n^{\lambda}1_{\U}],
\end{equation*}
for some constant $c$ uniform over all $\lambda$, and not depending
on $k$. We will show
\begin{equation*}
\E^{\ogp_0}[Z_n^{\lambda}1_{\U}]\ge\E^{\og_0}[Z_n^{\lambda}1_{\U}][1-O(e^{-\beta
k})],
\end{equation*}
where
$\displaystyle{\beta=\lambda_1\min\{\alpha',\alpha'',\frac{1}{2}\log\frac{p}{1-p}\}}$.
There are two cases to consider. If $\lambda\le 1$, raising
\eqref{bound} to $\lambda$ and taking expectations,
\[\E^{\og_0}[Z_n^{\lambda}1_{\U}]
\le \E^{\ogp_0}[Z_n^{\lambda}1_{\U}]+
\left[c\left(\frac{1-p}{p}\right)^{\lambda k}+ e^{-\lambda\alpha''
k}\right]\E^{\og_0}[Z_n^{\lambda}1_{\U}]
\]
When $\lambda>1$, using the Minkowski inequality,
\[
\E^{\og_0}[Z_n^\lambda 1_{\U}]^{1/\lambda}\le
\E^{\ogp_0}[Z_n^\lambda 1_{\U}]^{1/\lambda}+\left[
c^{1/\lambda}\left(\frac{1-p}{p}\right)^k+ e^{-\alpha''
k}\right]\E^{\og_0}[Z_n^\lambda 1_{\U}]^{1/\lambda}
\]
Collecting terms and raising to $\lambda$,
\[\E^{\ogp_0}[Z_n^\lambda 1_{\U}]\ge
\E^{\og_0}[Z_n^\lambda
1_{\U}]\left[1-c^{1/\lambda}\left(\frac{1-p}{p}\right)^k-
e^{-\alpha'' k}\right]^{\lambda}\ge
\E^{\og_0}[Z_n^{\lambda}1_{\U}][1-O(e^{-\beta k})].
\]
We conclude that
\begin{equation}\label{good}
\E^{\og_0}[Z_n^{\lambda} 1_{\U}]\le
\E^{\ogp_0}[Z_n^{\lambda}][1+O(e^{-\beta k})]
\end{equation}

Now, conditioning on paths satisfying $\hg_n \cap
G_{-k/2}\neq\emptyset$, and using Lemma \ref{separation2}, we can
find a constant $c$, uniform over $\lambda$ and independent of $k$,
such that
\begin{equation}\label{bad}
\E^{\og_0}[Z_n^{\lambda}|\hg_n \cap
  G_{-k/2}\neq\emptyset]\le c\E^{\og_0}[Z_n^{\lambda}]
\end{equation}
Using equations \eqref{good} and \eqref{bad}, we get
\begin{eqnarray*}
\E^{\og_0}[Z_n^{\lambda}]
&=&\E^{\og_0}[Z_n^{\lambda} 1_{\U}]+\E^{\og_0}[Z_n^{\lambda};\hg_n \cap G_{-k}\neq\emptyset]\\
&\le& [1+O(e^{-\beta k})]\E^{\ogp_0}[Z_n^{\lambda}]
+c\,\left(\frac{1-p}{p}\right)^k \E^{\og_0}[Z_n^{\lambda}]\\
&=&[1+O(e^{-\beta k})]\E^{\ogp_0}[Z_n^{\lambda}],
\end{eqnarray*}
with the last step coming from $\beta\le
\frac{1}{2}\log\left(\frac{p}{1-p}\right)$. The proposition follows
immediately from the definition of $K_n(\og_0)$.
\end{proof}

\begin{corollary}\label{c1estimate}
For all $k$ and all $\og_0, \ogp_0 \in V_k$ with $\og_0=_k\ogp_0$,
and $\gamma_1 \in W_{1,-k/2}$,
\begin{equation*}\label{c1}
e^{-\lambda\Phi (\tg'_1)}\ge e^{-\lambda\Phi
  (\tg_1)}[1-O(e^{-\beta k})].
\end{equation*}
\end{corollary}

\begin{proof} Since $\tg_0$ and $\tg'_0$ have the same endpoint,
we can attach to both of them the same element of $\mathcal{X}$,
namely $\gamma_1$. Let $\tg_1$ and $\tg'_1$ be the resulting paths
translated accordingly. From \eqref{bound}, using Lemma \ref{lemA},
in the case $n=1$ we get:
\[Z_1(\tg'_1)\ge Z_1(\tg_1)
\left[1-e^{-\alpha'' k}- \frac{d}{p}\left(\frac{1-p}{p}\right)^k
\right]\] Taking $\beta$ as in the previous proposition, it is easy
to see that for all $\lambda\in [\lambda_1,\lambda_2]$,
\[\left[1-e^{-\alpha'' k}-
\frac{d}{p}\left(\frac{1-p}{p}\right)^k \right]^{\lambda}\ge
1-O(e^{-\beta k})\] and the corollary follows.
\end{proof}

Then, using our results above, for all $\gamma_1 \in W_{1,-k/2}$,
the chain has the following property: the infimum over all
$\og_0=_k\og'_0$ in $V_k$,
\begin{equation}\label {diagbound}
\inf \frac{Q_N(\tg_1|\tg_0)}{Q_N(\tg'_1|\tg'_0)}= \inf\left [
\frac{e^{-\lambda\Phi (\og_1)}}{e^{-\lambda\Phi (\ogp_1)}} \cdot
\frac{K_{N-1}(\og_1)}{K_{N-1}(\ogp_1)}\cdot \frac{K_N(\ogp_0
)}{K_N(\og_0)}\right ] =1-O(e^{-\beta k}) \ge 1-c\, e^{-\beta k},
\end{equation}
for some constant $c$ uniform over $\lambda$. This will be the main
estimate used in the proof of Theorem \ref{expc}.

\subsection{Coupling of weighted paths}\label{sec:coup}

Given a double history $(\og_0,\ogp_0)$, we let $X_0=\og_0$ and
$X'_0=\ogp_0$ be the starting configurations of our
time-inhomogeneous Markov chains. In this section, we define $X_n$
and $X'_n$ on the same probability space $(\tilde{\Omega},\tilde{
\mathcal{F}}, \tProb)$, we will show there is a good chance the
paths will couple in two steps, and once coupled, they will remain
coupled for another step with positive probability. Furthermore, if
the paths are coupled for $k$ steps, we will show the paths will
decouple at a rate that decays exponentially in $k$. The key tool
used in proving the exponential rate of decay is our main estimate
from the previous section, equation \eqref{diagbound}.

On $V_k$, we will use a maximal coupling for $X_n$ and $X'_n$. It is
essentially the same coupling as the one described in \cite{Coupling
Paper}. For $n\ge k+2$, if $\tg_n=_k\tg'_n$, we say the chains $X$
and $X'$ have been coupled for $k+1$ steps by level $n+1$, denoted
by $X_{n+1}\equiv_{k+1}X'_{n+1}$, if:
\begin{itemize}
\item $X_n\equiv_k X'_n$, that is, $X$ and $X'$ have been coupled for
$k$ steps by level $n$,
\item  $\gamma_{n+1}=\gamma'_{n+1}$,
\item $\tg_{n}, \tg'_{n} \in V_{k}$
\end{itemize}
We think of this coupling in the following way: suppose $X_n$ and
$X'_n$ have been coupled for $k$ steps; if $\tg_n$ does not have
enough connected cross-sections, we decouple, otherwise, the chains
couple for an additional step if $\gamma_{n+1}=\gamma'_{n+1}$.
\begin{remark}
Note that when we say $X_n$ and $X'_n$ are coupled for the last $k$
steps, we do not mean that $X_j$ is equal to $X'_j$ for $n-k+1\le
j\le n.$ What we mean is that for $X_n=\tg_n$ and $X'_n=\tg'_n$ to
be coupled for the last $k$ steps, we simply need
$\gamma_j=\gamma'_j$ for $n-k+1\le j\le n.$ An equivalent way to set
up this problem is to construct a chain $\tilde{X}_n$ with history
$\og_0$, on one-level paths from $\mathcal{X}$: for $1\le n$, let
$\tilde{X}_j=\gamma_j$, with transition probabilities as in
\eqref{transitions}. These chains would be non-Markovian,
time-inhomogeneous and dependent on the initial configurations, but
they would couple in the classical sense, that is, we would say
$\tilde{X}_n$ and $\tilde{X}'_n$ are coupled if
$\tilde{X}_n=\tilde{X}'_n$. We prefer our setup for notation
purposes only and the reader is welcome to think of the coupling in
terms of $\tilde{X}_n$ if so wishes.
\end{remark}
Let $\sigma_n=\sigma(\tg_n,\tg'_n)$, be the minimum number of steps
backward that are needed to find a difference in coupling. On the
event $\{\sigma_n=k\}$, either the paths decouple or they couple for
an additional step, and thus $\sigma_{n+1}\in \{0,k+1\}.$ We define
the following family of transition
 probabilities $q_n$ for the triple $(X_n, X_n'; \sigma_n)$:
if $\gamma_{n+1}=\gamma'_{n+1},$
\[q_{n+1}(\tg_{n+1},\tg'_{n+1}; k+1\vert
\tg_{n},\tg'_{n};k)=Q_{N}(\tg_{n+1}|\tg_n)\wedge
Q_{N}(\tg'_{n+1}|\tg'_n)1_{\{\tg_n, \tg'_n\in V_k\}}\]
If
$\gamma_{n+1}\neq\gamma'_{n+1},$ the transition probability
$q_{n+1}(\tg_{n+1},\tg'_{n+1}; 0\vert \tg_{n},\tg'_{n};k)$ could
also be positive. One can give a formula for this case, but since we
will not use it, we refer the reader to the maximal coupling
presented in \cite{Coupling Paper} for details. Then these
transition probabilities define a coupling of $X_n$ and $X'_n$. We
use $\tProb$ as shorthand for $\tProb^{\og_0,\ogp_0}$. The two
chains decouple at the $(n+1)$-th step if either we attach different
elements of $\mathcal{X}$ to $\tg_n$, and $\tg'_n$ respectively, or
if $\tg_{n}$ and implicitly $\tg'_{n}$ do not have enough connected
cross-sections.

We would like to estimate $\tProb\{\sigma_{n+1}=k+1\vert
\sigma_n=k\}$. First observe that given $\tg_n \in V_k$,
\begin{equation}\label{bad set}
\sum_{\gamma_{n+1} \notin W_{n+1,n-k/2}}Q_{N}(\tg_{n+1}\vert
\tg_n)\le O\left[\left(\frac{1-p}{p}\right)^{k/2}\right]=O(e^{-\beta
k}),
\end{equation}
and thus, given two histories $\tg_n$ and $\tg'_n$ in $V_k$, and
using $\eqref{diagbound}$, and $\eqref{bad set}$,
\begin{eqnarray*}
\tProb\{\sigma_{n+1}=k+1\vert \tg_n, \tg'_n;\sigma_n=k\}&\ge& \sum
Q_{N}(\tg_{n+1}\vert \tg_n)\wedge Q_{N}(\tg'_{n+1}\vert
\tg'_n)\\
&\ge& \sum Q_{N}(\tg_{n+1}\vert \tg_n)\left[1-O(e^{-\beta
    k})\right]\\
    &\ge& 1-O(e^{-\beta k}),
\end{eqnarray*}
where the summation above is taken over all $\gamma_{n+1} \in
W_{n+1,n-k/2}$. Hence, on $V_k$, the chains will decouple with
probability less than $O(e^{-\beta k})$. Taking expectations over
$\hg_n$, and recalling our earlier result from Lemma
$\ref{largedev}$, we can find a constant $c$ such that
\begin{equation}\label{step3}
\tProb\{\sigma_{n+1}=0\vert \sigma_n=k\}\le O(e^{-\beta
  k})+O(e^{-\alpha' k})\le ce^{-\beta k}.
\end{equation}

Suppose $k\ge 1$ and $\tg_n=_k\tg'_n$ and the chains have been
coupled for $k$ steps. First assume $\tg_n\in V_k$ and let
$\gamma_{n+1}$ be a step forward in the $\Z$ direction of $G$. Since
$\og_n$ and $\og'_n$ have the same endpoint, we attach
$\gamma_{n+1}$ to both. Consider the $h$-process with the property
that it reaches level $n+1$ in exactly one step from level $n$. Then
one can show $Q_N(\tg_{n+1}|\tg_n)\ge c (p/d)^{1+\lambda}$ and thus
there exists a constant $c>0$ such that for all $k\ge 1$, for all
$\tg_n=_k\tg'_n\in V_k$,
\[\tProb\{\sigma_{n+1}=k+1\vert \tg_n, \tg'_n;\sigma_n=k\}\ge
Q_{N}(\tg_{n+1}\vert \tg_n)\wedge Q_{N}(\tg'_{n+1}\vert \tg'_n)\ge
c\left(\frac{p}{d}\right)^{1+\lambda}\] Taking expectations over all
pairs $\tg_n=_k\tg'_n$, for all $k\ge 1$,
\[
\tProb\{\sigma_{n+1}=k+1\vert \sigma_n=k\}\ge c
\left(\frac{p}{d}\right)^{1+\lambda}\Q^{\tg_0}_N\{\tg_n\in
V_k|\tg_{n-1}\in V_{k-1}\}\] Reasoning as above, if we start with a
half-infinite path with enough connected cross-sections, and we
attach an additional step, with $\Q_N$-probability greater than
$c(p/d)^{\lambda+1}$ we get a half-infinite path that has enough
connected cross-sections. Therefore, for all $k\ge 1$
\begin{equation}\label{step2}
\tProb\{\sigma_{n+1}=k+1\vert \sigma_n=k\} \ge\left(c
\left(\frac{p}{d}\right)^{1+\lambda}\right)^2.
\end{equation}
 Observe that even when
$k=0$, if $\og_n$ and $\ogp_n$ have the same endpoint, the paths
will get coupled on the next step with probability greater than $c
\left(\frac{p}{d}\right)^{1+\lambda}.$ Suppose $n\ge 2$. Then
\begin{equation*}
\tProb\{\sigma_{n+1}=1\vert \sigma_n=0\}\ge c
\left(\frac{p}{d}\right)^{1+\lambda}\tProb\{(n,0)\in \og_n\cap
\ogp_n\}.
\end{equation*}
We proceed to find a positive lower bound for $\tProb\{(n,0)\in
\og_n\cap \ogp_n\}.$ Suppose the chains have evolved up to level
$n-2$. Let $\gamma_{n-1}$ be given by taking one step forward in the
$\Z$ direction of $G$. Define $\gamma'_{n-1}$ in the same way. Let
$\gamma_n$ be the following path: take the shortest path to
$(n-1,0)$ by moving only on level $n-1$ of the cylinder, then take a
step forward in the $\Z$ direction of $G$. Define $\gamma'_n$ in the
same way. Then the 2-step paths $\gamma_{n-1}\gamma_n$ and
$\gamma'_{n-1}\gamma'_n$ each have a random walk measure bounded
from below by $\left(\frac{p}{d}\right)^2a$. Attach them to
$\og_{n-2}$ and $\og_{n-2}$ respectively. Now we have $\og_n$ and
$\ogp_n$ ending at the same point, namely $(n,0)$.

Observe that for any choice of $\hg_{n-2}$, if we attach paths
$\gamma_{n-1}\gamma_n$ as described above,
$\displaystyle{Q_N(\tg_n|\tg_{n-2}) \asymp
e^{-\lambda\Phi_2(\tg_2)}}$. It is easy to see that this quantity is
bounded from below by a positive constant: let $x$ be the point on
the torus with coordinates $(\lfloor L/2\rfloor,0,\dotsc,0)$ and
$x'$ be the point on $\mathcal{T}$ with coordinates $(\lfloor
L/2\rfloor+1,0,\dotsc,0)$. Consider the following path: if the
endpoint of $\og_{n-2}$ is equal to $(n-2,x)$, then $S$ takes the
shortest path to $(n-2,x')$, conditioned to avoid $(n-2,x)$,
otherwise $S$ takes the shortest path to $(n-2,x)$, conditioned to
avoid the the endpoint of $\og_{n-2}$; then $S$ moves two steps
forward in the radial direction. Observe that our choices for $x$
and $x'$ assure $S[T_{n-2},T_n]$ avoids $\og_n$. This occurs with
probability greater than $a\left(\frac{p}{d}\right)^2$, so we can
bound $Q_N(\tg_n|\tg_{n-2})$ from below by a constant which we can
choose uniformly over $\lambda$. Call this constant $c'.$ Then
\begin{equation}\label{step2'}
\tProb\{\sigma_{n+1}=1\vert \sigma_n=0\}\ge c \left
(\frac{p}{d}\right)^{1+\lambda}(c')^2.
\end{equation}
Let $b=\min\left\{c^2 \left(\frac{p}{d}\right)^{2(1+\lambda_2)},c
\left(\frac{p}{d}\right)^{1+\lambda_2}(c')^2\right\}$. From
\eqref{step2} and \eqref{step2'}, for all $n\ge 2$, for all $k\ge 0$
\[\tProb\{\sigma_{n+1}=k+1\vert \sigma_n=k\}\ge b .\]

\subsection{Proof of Theorem \ref{expc}.}
Using the setup from the coupling above, Theorem \ref{expc} can be
restated in the following way:

\begin{theorem}\label{t2}
Suppose $\sigma_1, \sigma_2,\cdots$ are non-negative integer valued
random variables adapted to a filtration
$\mathcal{F}_1\subset\mathcal{F}_2\subset \cdots$. Suppose there
exist positive constants $c, \alpha, b$ such that
\begin{itemize}
\item on the event $\{\sigma_n=k\}$, $\sigma_{n+1}\in \{0,k+1\},$
\item for all $k$ and all $n\ge 2$, $\tProb\{\sigma_{n+1}=k+1\vert
\mathcal{F}_n\}\ge b 1_{\{\sigma_n=k\}},$ \item for all $k$ and all
$n$, $\tProb\{\sigma_{n+1}=k+1\vert \mathcal{F}_n\}\ge
\left(1-ce^{-\beta k}\right) 1_{\{\sigma_n=k\}}.$
\end{itemize}
Then there exist constants $C$ and $\beta_1$ such that for all $n$
 \[\tProb \{\sigma_{2n} < n \} \le C e^{-\beta_1 n}.\]
\end{theorem}
\begin{proof}
Let $k_0$ be large enough so that  $1-ce^{-\beta k}>0$ for $k\ge
k_0$. Then we choose $\alpha$ small enough so that $1-b\le
e^{-\alpha k_0}$ and $ce^{-\beta k}\le e^{-\alpha(k+1)}$ for $k\ge
k_0$. Note that $(\sigma_n)$ is not a Markov chain, so consider the
following setup: let $s_n$ be a non-negative integer valued Markov
chain with $s_0=0$, and whose transition probabilities are given by:
\begin{align*}
p_{k,k+1}&=1-e^{-\alpha (k+1)}& p_{k,0}&=e^{-\alpha (k+1)}
\end{align*}
$s_n$ stochastically dominates $\sigma_{n+2}$. We claim that for
each $n\ge 0$,
\begin{equation}\label{dom}
\tProb\{s_n\ge k\}\le \tProb\{\sigma_{n+2}\ge k\}, \mbox{ for all
 $k\ge 0$.}
\end{equation}
The proof is done by induction on $n$. Note that \eqref{dom} holds
for all $n$, when $k=0$. We assume $k\ge 1$. For $n=0$, since
$\tProb\{s_0=0\}=1$, it follows that $\tProb\{s_0\ge k\}=0 \le
\tProb\{\sigma_2\ge k\}$. Assume $\tProb\{s_n\ge k\}\le
\tProb\{\sigma_{n+2}\ge k\}$ for all $k\ge 1$. Then
\begin{eqnarray*}
\tProb\{\sigma_{n+3}\ge k\}
&=&\sum_{j=k-1}^{\infty}\tProb\{\sigma_{n+3}= j+1\}\\
&=&\sum_{j=k-1}^{\infty}\tProb\{\sigma_{n+3}= j+1|\sigma_{n+2}= j\}\tProb\{\sigma_{n+2}= j\}\\
&\ge&
\sum_{j=k-1}^{\infty}(1-e^{-\alpha(j+1)})\left(\tProb\{\sigma_{n+2}\ge
j\}-\tProb\{\sigma_{n+2}\ge j+1\}\right)\\
 &=&(1-e^{-\alpha
k})\tProb\{\sigma_{n+2}\ge k-1\}+\sum_{j=k}^{\infty}(e^{-\alpha
j}-e^{-\alpha(j+1)})\tProb\{\sigma_{n+2}\ge j\}
\end{eqnarray*}
Using our inductive assumption, and then taking the same steps as
above, but backwards,
\begin{eqnarray*}
\tProb\{\sigma_{n+3}\ge k\} &\ge& (1-e^{-\alpha k})\tProb\{s_n\ge
k-1\}+\sum_{j=k}^{\infty}(e^{-\alpha j}-e^{-\alpha (j+1)})\tProb\{s_n\ge j\}\\
&=&\sum_{j=k-1}^{\infty}(1-e^{-\alpha(j+1)})\left(\tProb\{s_n\ge
j\}-\tProb\{s_n\ge j+1\}\right)\\
&=& \tProb\{s_{n+1}\ge k\}.
\end{eqnarray*}
We prove the theorem by showing there exist constants $C$ and
$\beta_1$ such that for all $n \in \mathbb{N}$,
\[\tProb \{ s _{2n-2} \ge n \} \ge 1-Ce^{-\beta_1 n}.\]
Let $\tau=\inf\{ n\ge 1: s_n=0\}$. Then
$\tProb\{\tau=1\}=e^{-\alpha}$, $\tProb\{\tau=k+1\}=e^{-\alpha
(k+1)}\prod_{j=1}^{k}(1-e^{-j\alpha})$ for $k\ge 1$, and
$\tProb\{\tau=+\infty\}=\prod_{j=1}^{\infty}(1-e^{-j\alpha})$.
Relating these stopping times to $s_n$, for $n\ge 2$ we have
\[\tProb\{s_n=0\}=\sum_{k=1}^n\tProb\{\tau=k\}\tProb\{s_{n-k}=0\}.\]
Define the following function
\begin{equation*}\label{F}
F(s)=\sum_{n=1}^{\infty}\tProb \{\tau=n\}s^n .
\end{equation*}
The radius of convergence of $F(s)$ is $e^{\alpha}$, and since
$F(1)=\tProb\{\tau<\infty\}<1$, by continuity of $F$, we can find
some $s^* \in (1, e^{\alpha})$ for which $F(s^*)=1$. Then for all $s
\in (1,s^*)$, $F(s)<1$ and
\[\sum_{n=0}^{\infty}\tProb
\{s_n=0\}s^n =\frac{1}{1-F(s)}<\infty .\] Thus, for all $s \in (1,
s^*)$ and $n$ large enough $\tProb\{s_n=0\}\le s^{-n}$. Then clearly
we can find constants $c$ and $\beta_1$ such that
$\tProb\{s_n=0\}\le c\, e^{-\beta_1 n}$ for all $n$. So,
\[\sum_{k=0}^{n-1}\tProb\{ s _{2n-2}=k \}
\le \sum_{k=0}^{n-1}\tProb\{ s _{2n-2-k}=0 \} \le
\sum_{k=0}^{n-1}c\, e^{-\beta_1(2n-2-k)} =c\,e^{-\beta_1
(2n-2)}\left(\frac{e^{\beta_1 n}-1}{e^{\beta_1}-1}\right)\] and the
theorem follows, with a constant $C=ce^{2\beta_1}(e^{\beta_1}
-1)^{-1}$.
\end{proof}

\subsection{Invariant measure: proof of Theorem \ref{invariant}}\label{sec:Invariant Measure}
Let $\Lambda$ denote the space of measures supported on $\oA$. Let
$L_{\lambda}: \Lambda\to \Lambda$ denote the following
transformation:
\[L_{\lambda}\nu(\tg_1)=\nu(\tg_0)e^{-\lambda\Phi(\tg_1)}\]
$L$ is linear and continuous on $\Lambda$. Furthermore, the $n$-th
iterate of $L_{\lambda}$ is, as expected, given by the expression:
\[L_{\lambda}^n\nu(\tg_n)=\nu(\tg_0)e^{-\lambda\Phi_n(\tg_n)}.\]
Suppose we start with a distribution $\nu$ on $\oA$, and we attach
an $n$-level path according to the probability measure $\Q_n$, then
re-scale the path accordingly to obtain an element of $\oA$. Then
the measure of this new path is given by  $\nu^n$, whose density
with respect to $\nu$ is
\begin{equation}\label{nun} \frac{e^{-\lambda
\Phi_n}}{\E^{\nu}[e^{-\lambda \Phi_n}]} \end{equation} Note that
$\nu^n$ is simply $L_{\lambda}^n\nu$ normalized to a probability
measure and it depends on the starting distribution $\nu$ and on
$\lambda$. Let $\nu^n_{k}$ be restriction of $\nu^n$ to the last $k$
steps of $\eta$. Then
\[\nu^n_k(\eta)=\Q_n^{\nu}\{X_n=_k\eta\}=\frac{\E^{\nu}[e^{-\lambda
\Phi_n}; \tg_n=_k\eta]}{\E^{\nu}[e^{-\lambda \Phi_n}]}.\]

We are interested in convergence of $\nu^n$ as $n\to\infty$. So let
us first consider $\nu$ to be the Dirac measure $\delta_{\og_0}$ and
define measures $\pi_k$ on $\oA$ as follows: for all $\eta$ in
$\oA$, let
\[\pi_k(\eta)
=\lim_{N\to \infty}\Q_N^{\tg_0}\{X_N=_k \eta\}.\] Note that $\pi_k$
is a measure on the restriction of $\eta$ to the last $k$ steps. To
show that the measures $\pi_k$ are well-defined and this limit
exists, first recall that given $X_j=\tg_j$, the law of $X_{N+j}$
under $\Q_{N+j}^{\tg_0}$ is the same as the law of $X_{N}$ under
$\Q_{N}^{\tg_j}$. Then for all $k< N/2$ and all $j\ge 0$, the
coupling result from Theorem \ref{expc} implies
\begin{equation}\label{jump}
\begin{split}
|\Q_N^{\tg_0}\{X_N&=_k \eta\}-\Q_{N+j}^{\tg_0}\{X_{N+j}=_k
\eta\}| \\
&=
|\sum_{\hg_j}\Q_N^{\tg_0}\{X_j=\tg_j\}\left(\Q_N^{\tg_0}\{X_N=_k\eta\}-\Q_N^{\tg_j}\{X_{N}=_k\eta\}\right)|\\&\le
\sum_{\hg_j}\Q_N^{\tg_0}\{X_j=\tg_j\}\left|\Q_N^{\tg_0}\{X_N=_k\eta\}-\Q_N^{\tg_j}\{X_{N}=_k\eta\}\right|\\
&\le \sum_{\hg_j}\Q_N^{\tg_0}\{X_j=\tg_j\}(C e^{-\beta_1 N})\\
&\le C e^{-\beta_1 N}
\end{split}
\end{equation}
Then clearly, for a given history $\tg_0 \in \oA$,
$\Q_N^{\tg_0}\{X_N=_k \eta\}$ converges as $N \to \infty$. Now, if
we start with a different history $\tg'_0$, from our coupling, and
more precisely from Corollary \ref{coupling}, we get
\[\lim_{N\to \infty}\Q_N^{\og_0}\{X_N=_k \eta\}=\lim_{N\to \infty}\Q_N^{\ogp_0}\{X'_N=_k\eta\}.\]
Therefore, the measures $\pi_k$ are well-defined. Moreover, if we
start with some other initial distribution $\nu$ on paths $\tg_0$
from $\oA$, the following holds
\begin{equation}\label{e3_1}
\sum_{\eta|_k}\left|\Q_N^{\nu}\{X_N=_k\eta)\}-\pi_k(\eta)\right| =O(
e^{-\beta_1 N}).
\end{equation}
The measures $\pi_k$ are consistent and then, by the Kolgomorov
Extension Theorem, they converge to a limiting measure on $\oA$,
which we will denote by $\displaystyle{\pi=\lim_{k\to \infty}
\pi_k}$. It is easy to check that $\displaystyle{\pi=\lim_{n\to
\infty}\nu^n}.$

We claim that $\pi$ is the unique stationary measure for the Markov
chains described in this section. The result follows from the next
proposition.

\begin{proposition} There exists $\beta_2>0$ such that for any starting
  distribution $\nu$, we can find a constant $c(\nu)$ so that for all
  $n\ge 0$ the following holds
\begin{equation}\label {est1}
\E^{\nu}[e^{-\lambda \Phi_n}]=c(\nu)e^{-\xi(\lambda)n}[1+
O\left(e^{-\beta_2 n}\right)]
\end{equation}
\end {proposition}

\begin{proof}
Let $a_n=\E^{\nu}[e^{-\lambda \Phi_n}]$. Then is it a quick check
that
\begin{eqnarray*}
\E^{\nu}[e^{-\lambda \Phi_{n+m}}]
&=&\E^{\nu}[e^{-\lambda \Phi_n}]\E^{\nu^n}[e^{-\lambda \Phi_m}].
\end{eqnarray*}
Let $\tilde{\Phi}$ depend only on the last $n/2$ steps of the path
$\tg_n$, that is, let $\tg'_n=\tg_n\vert_{n/2}$ be the restriction
of $\tg_n$ to its last $n/2$ steps and then
\[e^{-\lambda \tilde{\Phi}(\tg_n)}=e^{-\lambda \Phi(\tg'_n)}
.\] Suppose $\tg_n$ has enough connected cross-sections
  above level$-n/2$, more precisely, $\tg_n, \tg'_n$ are in $V_{n/2}$.
By Proposition \ref{cnestimate},
$\displaystyle{\E^{\tg_n}[e^{-\lambda\Phi}]=\E^{\tg_n}[e^{-\lambda\tilde{\Phi}}][1+O(e^{-\beta
n/2})]}$, and so
\begin{equation}\label{e1}
\begin{split}
\vert\E^{\nu^n}[e^{-\lambda\Phi};V_{n/2}]&-\E^{\pi}[e^{-\lambda\Phi};V_{n/2}]\vert\\
&\le \left[1+O(e^{-\beta
n/2})\right]\left\vert\E^{\nu^n_{n/2}}[e^{-\lambda\tilde{\Phi}};V_{n/2}]-\E^{\pi_{n/2}}[e^{-\lambda\tilde{\Phi}};V_{n/2}]\right\vert\\
 &\le \left[1+O(e^{-\beta n/2})\right]\Vert\nu^n_{n/2}-\pi_{n/2}\Vert.
\end{split}
\end{equation}
On the complement of $V_{n/2}$, which we will denote by $V_{n/2}^c$,
we will bound $\displaystyle{\E^{\tg_n}[e^{-\lambda \Phi}]}$ by $1$.
From Lemma \ref{largedev}, we know that
\begin{equation}\label{e2}
\nu^n(V^c_{n/2})=\nu^n_{n/2}(V^c_{n/2})=O( e^{-\alpha' n/4}).
\end{equation} Observe that \eqref{e3_1} implies
$\displaystyle{\Vert\nu^n_{n/2}-\pi_{n/2}\Vert= O(e^{-\beta_1
  n})}$, and thus,
\begin{equation}\label{e3}
\pi(V_{n/2}^c)=\pi_{n/2}(V_{n/2}^c)=O(e^{-\beta_1
  n})+O(e^{-\alpha' n/4}).
  \end{equation}
Combining estimates \eqref{e1}, \eqref{e2} and \eqref{e3},
\begin{equation*}
\left\vert\E^{\nu^n}[e^{-\lambda\Phi}]-\E^{\pi}[e^{-\lambda\Phi}]\right\vert
= O(e^{-\beta_1 n})+O(e^{-\alpha' n/2}).
\end{equation*}
Let $\beta_2=\min\{\beta_1, \alpha'/4\}$, and since
$\E^{\pi}[e^{-\lambda\Phi}]\asymp e^{-\xi(\lambda)}$ and thus it is
bounded from below by a positive constant, we get
\[\left\vert\E^{\nu^n}[e^{-\lambda\Phi}]-\E^{\pi}[e^{-\lambda\Phi}]\right\vert=O(e^{-\beta_2
n})\E^{\pi}[e^{-\lambda\Phi}].\] Note that our error term depends on
the initial measure on paths, namely $\nu$. Thus,
\[\displaystyle{a_n=a_0\prod_{j=0}^{n-1}\frac{a_{j+1}}{a_j}
=\prod_{j=0}^{n-1}\E^{\pi}[e^{-\lambda\Phi}] [1+O(e^{-\beta_2 j})]
=c(\nu) \, \left(\E^{\pi}[e^{-\lambda\Phi}]\right)^n[1+O(e^{-\beta_2
n})]}.\] We know that $a_n\asymp e^{-\xi(\lambda)n}$, and therefore
we must have
\[\E^{\pi}[e^{-\lambda\Phi}]=e^{-\xi(\lambda)},\]
which completes the proof of the proposition.
\end{proof}

We want to show $\pi$ is invariant under the linear transformation
$L_{\lambda}$. Re-writing \eqref{nun} in this notation, we obtain
\[\nu^n=\frac{L_{\lambda}^n \nu}{\E^{\nu}[e^{-\lambda\Phi_n}]}.\]
Starting with any distribution $\nu$ as the initial measure on
paths, we have seen that
\[L_{\lambda}\nu^n=\frac{L_{\lambda}^{n+1}\nu}{\E^{\nu}[e^{-\lambda\Phi_n}]}
=\nu^{n+1}\frac{\E^{\nu}[e^{-\lambda\Phi_{n+1}}]}{\E^{\nu}[e^{-\lambda\Phi_n}]}=
e^{-\xi(\lambda)}\nu^{n+1}[1+O(e^{-\beta_2 n})].\] Passing to the
limit, as $n\to \infty$, from continuity of $L_{\lambda}$ we have
$L_{\lambda}\pi=e^{-\xi(\lambda)}\pi$ and therefore for all $n\ge
0$, $\pi^n=\pi$. Note that $\pi$ depends on $\lambda$.

\section{Analyticity of intersection exponents}\label{sec:anal}

In this section we prove that $\xi(1,\lambda)$ is a real analytic
function of $\lambda$ for $\lambda>0$. The proof for other values of
$k$ is essentially the same. Key to this proof is convergence to the
unique invariant measure $\pi$, presented in the previous section,
and more significantly the exponential rate at which this
convergence takes place. For each $\lambda>0$, we associate to
$L_{\lambda}$ a linear functional defined as
\[T_{\lambda}^n f (\tg_0)=\E[f(\tg_n)Z_n^{\lambda}],\]
for continuous functions $f$, bounded under an appropriately chosen
norm. This norm will be chosen so that on the Banach space of
functions with finite norm $\lambda\mapsto T_{\lambda}$ is an
analytic operator-valued function. In particular, the functions on
this Banach space have the property that their dependence on the
behavior of paths far away in the past decays exponentially.
Estimates already obtained from convergence to invariant measure
will show $\xi(\lambda):=\xi(1,\lambda)$ is an isolated simple
eigenvalue for $T_{\lambda}$. Using results from operator theory,
for every $\lambda>0$ one can then extend $x\mapsto\xi(x)$ to an
analytic function in a neighborhood of $\lambda$. This immediately
proves that intersection exponent $\xi(\lambda)$ is real analytic in
$\lambda$.
\begin{remark}
We reiterate that this section follows the notation and proof
outlines from  \cite{Anal Paper} in which Lawler, Schramm and Werner
prove analyticity of 2-dimensional Brownian exponents. We include
here the full proofs, for the sake of completeness, with the
understanding that they differ from \cite{Anal Paper} only in the
estimates we use.
\end{remark}
\subsection{The operator}\label{Bspace}

Let $\mathcal{C}$ be the set of continuous functions $f:\oA\to
\mathbb{C},$ bounded under the uniform norm $\displaystyle{\Vert
f\Vert=\sup_{\tg_0}|f(\tg_0)|}$. We are interested in functions that
depend very little on how $\tg_0$ looks like near negative infinity. Recall that $\tg\equiv_k\tg'$
means the paths $\tg$ and $\tg'$ have been coupled for the last $k$
steps, in the sense of Section \ref{sec:coup}.
Thus, let us consider the following $u$-norm: for all $f\in
\mathcal{C}$, and $u>0$, let
\[\Vert f\Vert_u:=\max\{\Vert f\Vert, \sup\{e^{ku}|f(\tg)-f(\tg')|:
k=1,2,\cdots,  \tg\equiv_k\tg' \}\}.\] Recall that $\tg\equiv_k\tg'$
means the paths $\tg$ and $\tg'$ have been coupled for the last $k$
steps, in the sense of the coupling described in Section
\ref{sec:coup}. This norm similar to the one used in \cite{Anal
Paper}. Let $\mathcal{C}_u:=\{f\in \mathcal{C}:\Vert
f\Vert_u<\infty\}$ denote the Banach space of all bounded functions
$f$ under the norm $\Vert f\Vert_u$. Let $\mathcal{L}_u$ be the
Banach space of continuous linear operators from $\mathcal{C}_u$ to
$\mathcal{C}_u$ with the usual norm
\[N_u(T):=\sup_{\Vert f\Vert_u=1}\Vert T(f)\Vert_u.\]
For all $\lambda>0$, and all $n>0$, we define the linear operator
$T_{\lambda}:\mathcal{C}\to \R$ by
\[T^n_{\lambda} f(\tg_0):=\E^{\tg_0}\left[f(\tg_n)e^{-\lambda\Phi_n}\right],\]
where the expectation is over the randomness of $\hg_n$. It is easy
to see that $T_{\lambda}$ is a semigroup of operators:
\begin{eqnarray*}
T^{n+m}_{\lambda}f(\tg_0)&=&\E^{\tg_0}\left[f(\tg_{n+m})e^{-\lambda\Phi_{n+m}}\right]\\
&=&\E^{\tg_0}\left[e^{-\lambda\Phi_n}\E^{\tg_n}\left[f(\tg_{n+m})e^{-\lambda\Phi_m}\right]\right]\\
&=&\E^{\tg_0}\left[e^{-\lambda\Phi_n}T^m_{\lambda} f(\tg_n)\right]\\
&=& T^n_{\lambda}T^m_{\lambda}f(\tg_0).
\end{eqnarray*}
We will use $T_{\lambda}$ for $T^1_{\lambda}$. One can similarly
define $T^n_{\lambda}$ for complex $\lambda$ with $\Re(\lambda)>0.$

\subsection{Analyticity of operator}\label{Tanal}
If we look at the functional $T_{z}$ as a function of $z$, it is
analytic in a small neighborhood of the positive real line. This
will be the first step in proving $e^{-\xi(\lambda)}$ is analytic in
$\lambda.$

\begin{proposition}
If $\lambda_1\le\lambda\le\lambda_2$, there exist $\epsilon>0$ and
$v(\lambda)>0$ such that for all $u\in(0,v)$, $z\mapsto T_z$ is an
analytic function from $\{z:|z-\lambda|<\epsilon\}$ into
$\mathcal{L}_u$.
\end{proposition}
\begin{proof}
Fix $\lambda>0$. For all $k\ge 0$, and all $f\in \mathcal{C}$,
$\tg_0\in\mathcal{A}$, let
\[U_k f(\tg_0)=\E^{\tg_0}\left[f(\tg_1)\frac{\Phi^k}{k!}e^{-\lambda\Phi}\right]\]
An upper bound for $U_kf$ is
\begin{equation}\label{uk}
\left\vert U_k f(\tg_0)\right\vert\le\Vert
f\Vert\E^{\tg_0}\left[\frac{(\lambda\Phi)^k}{k!}\lambda^{-k}e^{-\lambda\Phi}\right]\le
\Vert f\Vert
\E^{\tg_0}\left[e^{\lambda\Phi}\lambda^{-k}e^{-\lambda\Phi}\right]\le
\Vert f\Vert \lambda^{-k},
\end{equation}
and by dominated convergence, for all $z\in\mathbb{C}$ with
$|z|<\lambda$,
\[T_{\lambda-z}f(\tg_0)=\sum_{k=0}^{\infty}U_kf(\tg_0)z^k.\]
We need to show there exists a $v(\lambda)>0$ such that for $u<v$,
for all $k$, the operator norm of $U_k$ in $\mathcal{L}_u$ is
bounded by $b^k$ for some $b>0$. Then for $|z|<b^{-1}$,
$T_{\lambda-z}f$ is an analytic function of $z$ into $\mathcal{L}_u$
and the proposition follows. We will now prove this claim. From
\eqref{uk}, we have $\Vert U_k\Vert\le \lambda^{-k}.$ Now suppose
$\tg_0\equiv_m\tg'_0$ (that is, $\tg_0$ and $\tg'_0$ are coupled for
$m$ steps). We use $\Phi$ to denote $\Phi(\tg_1)$ and $\Phi'$ for
$\Phi(\tg'_1)$.
\begin{equation}\label{udk}
\begin{split}
\left|U_kf(\tg_0)-U_kf(\tg'_0)\right|&\le
\E\left[|f(\tg_1)-f(\tg'_1)|\frac{\Phi^k}{k!}e^{-\lambda\Phi}\right]\\
&+\Vert
f\Vert\E\left[\left|\frac{\Phi^k}{k!}e^{-\lambda\Phi}-\frac{(\Phi')^k}{k!}e^{-\lambda\Phi'}\right|\right]
\end{split}
\end{equation}
Given $\tg_0\equiv_m\tg'_0$, on the event the two paths remain
coupled for an additional step, we can bound
$\displaystyle{|f(\tg_1)-f(\tg'_1)|}$ by $\Vert f\Vert e^{-(m+1)u}$
and otherwise we bound it by $2\Vert f\Vert$. Using our coupling
result in \eqref{step3}, an upper bound for the first term in
\eqref{udk} is
\begin{equation}
\begin{split}\label{u-good} \Vert f\Vert_u e^{-(m+1)u}\lambda^{-k}+2\Vert
f\Vert \lambda^{-k}\tilde{\Prob}\{\tg_1\not\equiv_{m+1}\tg'_1\vert
\tg_0\equiv_m\tg'_0\}\\ \le \Vert
f\Vert_u\lambda^{-k}\left(e^{-(m+1)u}+ce^{-m\beta}\right).
\end{split}\end{equation}
For the second term in \eqref{udk}, we have
\begin{eqnarray*}
\left(\frac{\lambda}{2}\right)^k
\E\left[\left|\frac{\Phi^k}{k!}e^{-\lambda\Phi}-\frac{(\Phi')^k}{k!}e^{-\lambda\Phi'}\right|\right]&=&\frac{1}{k!}
\E\left[\left\vert(\frac{\lambda\Phi}{2})^ke^{-\lambda\Phi}-(\frac{\lambda\Phi'}{2})^ke^{-\lambda\Phi'}\right\vert\right]\\
&\le& 3\E\left[\vert
e^{-\lambda\Phi/2}-e^{-\lambda\Phi'/2}\vert\right]. \end{eqnarray*}
Suppose that after being coupled for $m$ steps the paths remain
coupled for an additional step. We thus attach the same $\hg_1$ to
$\tg_0$ and $\tg'_0$, and we consider two cases: if $\hg_1 \in
G_{m/2}$, then using our estimate from \eqref{diagbound} we get
$\vert e^{-\lambda\Phi/2}-e^{-\lambda\Phi'/2}\vert\le
c\,e^{-m\beta}$. On the complement of $G_{m/2},$ as well as when the
paths decouple after $m$ steps, we will bound this difference by
$2$. Using the inequality $\displaystyle{\tilde{\Prob}\{\hg_1\notin
G_{m/2}\}\le e^{-\beta m}}$, an recalling our result from
\eqref{step3}, we obtain
\[3\E\left[\vert e^{-\lambda\Phi/2}-e^{-\lambda\Phi'/2}\vert\right]\le
c\, e^{-\beta m},\] with a different $c$, uniform in $\lambda$ and
independent of $k$. Thus, the second term in \eqref{udk} can be
bound by
\begin{equation}\label{factorial}
c\,\Vert f\Vert_u (\frac{2}{\lambda})^ke^{-m\beta}\, .
\end{equation}
From estimates \eqref{u-good} and \eqref{factorial}, for all $u\le
\beta$,
\[\left|U_kf(\tg_0)-U_kf(\tg'_0)\right|\le c\,\Vert f\Vert_u(\frac{2}{\lambda})^{k}e^{-mu}.\]
Hence $N_u(U_k)\le c\, (\frac{2}{\lambda})^k$ and the proposition
follows with $v(\lambda)\le\beta$ and $\epsilon\le\lambda/2$.
\end{proof}

\subsection{Analyticity of exponent}\label{sec:analyticity}
\begin{proposition}\label{eigen}
$e^{-\xi(\lambda)}$ is an isolated simple eigenvalue for
$T_{\lambda}.$
\end{proposition}

\begin{proof} Fix $\lambda>0$. For ease of notation, we write $T$ as shorthand for
$T_{\lambda}$ and $e^{-\xi}$ for $e^{-\xi(1,\lambda)}.$

First we will show $\displaystyle{\frac{T^n f(\tg_0)}{T^n
1(\tg_0)}}$ converges to a bounded functional $h$. Recall the result
of our coupling: for any  $\tg_0, \tg'_0\in \oA$,
 \[\tilde{\Prob}\{\tg_n\not\equiv_{n/2}
 \tg'_n\}\le Ce^{-\beta_1 n}.\]
  Suppose $\tg_0$ is fixed. If we let $\tg'_0=\tg_k$ then the law of
 $\tg'_0$ under $\Q_{n}$ is the same as the law of $\tg_k$ under $\Q_{n+k}$, and from
 the coupling result, for all $f\in \mathcal{C}_u$,
\begin{equation}\label{hlim}
\left\vert\frac{T^{n+k} f(\tg_0)}{T^{n+k}
    1(\tg_0)}
-\frac{T^n f(\tg_0)}{T^n 1(\tg_0)}\right\vert\le \int
    \vert f(\tg_n)-f(\tg'_n)\vert \,d\tilde{\Prob}\le 2\Vert f\Vert
    Ce^{-\beta_1 n}+\Vert f\Vert_u e^{-nu/2} \end{equation}
Therefore, $\displaystyle{\frac{T^n f(\tg_0)}{T^n
    1(\tg_0)}\to h(f,\tg_0)}$.
Similarly, for any starting configurations $\tg_0, \tg'_0\in \oA$
and all $f\in \mathcal{C}_u$
\[\left\vert\frac{T^{n} f(\tg_0)}{T^{n}
    1(\tg_0)}-\frac{T^n f(\tg'_0)}{T^n 1(\tg'_0)}\right\vert\le 2\Vert f\Vert
    Ce^{-\beta_1 n}+\Vert f\Vert_u e^{-nu/2}.\]
This shows $h$ is independent of $\tg_0$. It is easy to see that $h$
is a linear on $\mathcal{C}$ and $\Vert h\Vert\le \Vert f\Vert$.
Then $h$ is a bounded linear functional on $\mathcal{C}_u$.

Now we want to consider the functional $f \mapsto T^nf-h(f)T^n1$ and
to find an upper bound for its $N_u$ norm. From \eqref{hlim}, for
$u\le 2\beta_1$ and all $f\in \mathcal{C}_u$ and $\tg_0\in \oA$,
\begin{equation}\label{eq:norm bound}
\vert T^nf(\tg_0)-h(f)T^n1(\tg_0)\vert\le (2C+1)\Vert f\Vert_u
e^{-nu/2}T^n1(\tg_0)\le c\Vert f\Vert_u e^{-n(\xi+u/2)}.
\end{equation}
Then $\displaystyle{\Vert T^n(\cdot)-h(\cdot)T^n1\Vert\le
  c\,e^{-n(\xi+n/2)}}$. To find the $N_u$ norm, consider
$\tg_0\equiv_k\tg'_0$. For $k\le n/4$, the estimate above gives
\[\vert
T^nf(\tg_0)-h(f)T^n1(\tg_0)-T^nf(\tg'_0)+h(f)T^n1(\tg'_0)\vert\le
2c\Vert f\Vert_u e^{-n(\xi+u/4)}e^{-ku}.\]
When $k>n/4$, if $\tg_0$
and $\tg'_0$ are coupled for the last $k$ steps, we have shown in
Proposition \ref{cnestimate} that
$\displaystyle{T^n1(\tg_0)=T^n1(\tg'_0)\left[1+O(e^{-\beta
      k})\right]}$.
From our coupling, the two paths will remain coupled for $n$
additional steps with probability greater than
$\displaystyle{\prod_{j=0}^{n-1}\left[1-ce^{-\beta (k+j)}\right]}$
which can be shown to be bounded by $1-ce^{-\beta k}$, with a
different $c$, independent of $n.$ Hence for all $u\le\beta/2$ and
all $f\in \mathcal{C}_u$,
\[\left\vert\frac{T^{n} f(\tg_0)}{T^{n}
    1(\tg_0)}
-\frac{T^n f(\tg'_0)}{T^n 1(\tg'_0)}\right\vert\le 2c\,\Vert
f\Vert_u
   e^{-\beta k}+\Vert f\Vert_u e^{-(k+n)u}\le c'\,\Vert
   f\Vert_ue^{-nu/4}e^{-k u}\]
Multiplying this expression by $T^n1(\tg_0)$ and recalling that
$T^n1(\tg_0)\le ce^{-\xi n}$, we get
 \begin{equation*}
 \vert T^{n} f(\tg_0)-h(f)T^n1(\tg_0) -T^n f(\tg'_0) \frac{T^n
1(\tg_0)}{T^n 1(\tg'_0)}+h(f)T^n1(\tg_0) \vert \le c\,e^{-\xi
n}\Vert
   f\Vert_ue^{-nu/4-ku}
\end{equation*}
From \eqref{eq:norm bound} and Proposition \ref{cnestimate}
\begin{equation*}
\begin{split}
\vert
T^nf(\tg_0)-h(f)T^n1(\tg_0)&-T^nf(\tg'_0)-h(f)T^n1(\tg'_0)\vert\\
&\le c\,e^{-\xi n}\Vert
   f\Vert_ue^{-nu/4}e^{-k u}+O(e^{-\beta k})\Vert f\Vert_u
   e^{-n(\xi+u/2)}\\
&\le c'\,\Vert f\Vert_ue^{-n(\xi+u/4)}e^{-ku}.
\end{split}
\end{equation*}
We conclude that
\begin{equation}\label{eq:norm}
N_u(T^n(\cdot)-h(\cdot)T^n1)\le ce^{-n(\xi+u/4)}.
\end{equation}

This will show that $e^{-\xi}$ is a simple eigenvalue of $T$. Since
for all $k\ge 1$, $\displaystyle{\frac{T^{n+k}1}{T^n1}\to h(T^k1)}$
and for all $\tg_0$, $T^n1(\tg_0)\asymp e^{-\xi n}$, we get
$h(T1)=e^{-\xi}$ and $h(T^n1)=e^{-\xi n}$. From \eqref{eq:norm},
\[\Vert T^{n+k}1-h(T^k)T^n1\Vert_u\le ce^{-\xi(n+k)}e^{-nu/4},\]
Recall that $K_n(\tg_0)=e^{\xi n}T^n1(\tg_0)$ and then for all $k\ge
0$,
\[\Vert K_{n+k}-K_n\Vert_u\le ce^{-nu/4},\]
and therefore $K_n\to K$, for some function $K:\mathcal{A}\to
\mathbb{R}$. Furthermore, $\Vert K_n-K\Vert_u\le ce^{-nu/4}$. It
follows that $N_u(T^n(\cdot)-e^{-\xi n}h(\cdot)K)\le c\,e^{-\xi
n}e^{-nu/4}$. In particular, for all $f\in \mathcal{C}_u$,
\begin{equation}\label{b1} \Vert
T^n(f)-e^{-\xi n}h(f)K\Vert_u\le c\,e^{-n(\xi+u/4)}\Vert f\Vert_u.
\end{equation}
It is easy to see that for all $n\ge 1$,
\begin{equation}\label{hf}
h(T^nf)=e^{-\xi n}h(f).
\end{equation}
Moreover, from continuity of $T$,\begin{equation}\label{TK}
TK=\lim_{n\to\infty}TK_n=\lim_{n\to\infty}e^{\xi
n}T^{n+1}1=e^{-\xi}\lim_{n\to\infty}K_{n+1}=e^{-\xi}K,
\end{equation}
and hence $e^{-\xi}$ is an eigenvalue for $T$. By continuity and
linearity of $h$,
\begin{equation}\label{hK}
h(K)=\lim_{n\to\infty}h(K_n)=\lim_{n\to\infty}e^{\xi n}h(T^n1)=1.
\end{equation}
From estimates \eqref{hf}, \eqref{TK} and \eqref{hK}, one can easily
check that $T^n(\cdot)-e^{-\xi n}h(\cdot)K$ is the $n$-th iterate of
$T(\cdot)-e^{-\xi}h(\cdot)K$ and thus
$N_u(T(\cdot)-e^{-\xi}h(\cdot)K)\le e^{-\xi}e^{-u/4}$.

We claim this implies $e^{-\xi}$ is an isolated eigenvalue. We will
prove that for every $z$ with $1/2(1+e^{-u/4})<|z|<1$, there exists
$\epsilon>0$ such that for all $\Vert f\Vert_u=1$, $\Vert
e^{\xi}Tf-z f\Vert_u\ge \epsilon$, that is, $z$ is in the resolvent
set of $\tT=e^{\xi}T$.

Fix $z$ with $1/2(1+e^{-u/4})<|z|<1$, and for $\Vert f\Vert_u=1$ let
\[\tT f=z f+g \hspace{.5in} v_n(f)=\tT ^nf-h(f)K.\]
Note that $g \in \mathcal{C}_u$ and
\begin{equation}\label{ng}
\tT^nf=z^nf+\sum_{j=1}^nz^{n-j}\tT^{j-1}g.
\end{equation}
Since $K_n$ converges to $K$, by Proposition \ref{estconst} we have
$\Vert K\Vert_u\le c_2$. Recalling that $\Vert h(g)\Vert_u\le\Vert
g\Vert_u$ and using \eqref{b1}, we arrive at
\begin{eqnarray*}\label{gf}
\left\Vert \tT^nf-z^nf-\frac{1}{1-z}h(g)K\right\Vert_u &\le&
\left\Vert\sum_{j=1}^nz^{n-j}\tT^{j-1}g-\frac{1}{1-z}h(g)K\right\Vert_u\\
&=& \left\Vert\sum_{j=1}^nz^{n-j}v_{j-1}(g)\right\Vert_u+\left\Vert
\left(\sum_{j=1}^nz^{n-j}-\frac{1}{1-z}\right)h(g)K\right\Vert_u\\
&\le& c|z|^n\Vert g\Vert_u \, ,
\end{eqnarray*}
for some constant $c>1$ that depends on $z$ and $u$. Since this
bound holds for all $n$, and so does \eqref{b1}, we must have
$h(g)=(1-z)h(f)$. Therefore, for $f$ with $\Vert f\Vert_u=1$, and
for all $n$,
\[c|z|^n\Vert g\Vert_u\ge\Vert
\tT^nf-z^nf-h(f)K\Vert_u\ge|z|^n\Vert f\Vert_u-\Vert
v_n(f)\Vert_u\ge |z|^n-e^{-nu/4}.\] For $|z|>1/2(1+e^{-u/4})$, this
implies
\[\Vert g\Vert_u\ge\frac{1-e^{-u/4}}{c(1+e^{-u/4})}.\]
It follows that the spectrum of $T$ in $\mathcal{L}_u$ is the union
of $e^{-\xi}$ and a set contained in the ball of radius
$1/2(1+e^{-u/4})e^{-\xi}$ and centered at the origin.
\end{proof}

\begin{proofof}{\bf Proof of Theorem \ref{anal_proof}:}
We prove the theorem for $k=1$. From Proposition \ref{eigen},
$\xi(1,\lambda)$ is an isolated simple eigenvalue for the analytic
operator $T_{\lambda}$. By 4.16 in \cite{Wolf}, for all $\lambda>0$,
$x\mapsto\xi(1,x)$ can be extended analytically in a neighborhood of
$\lambda$. More precisely, for all $\lambda>0$, there exists
$\epsilon>0$ such that $z\mapsto\xi(1,z)$ is analytic in
$|z-\lambda|<\epsilon$. Therefore, piecing together these
$\epsilon$-balls we obtain a neighborhood of the positive real-line
$(0,\infty)$ where $z\mapsto\xi(1,z)$ is analytic.
\end{proofof}

\end{document}